\documentclass{birkjour}
\usepackage{amsthm}
\usepackage{amsmath}
\usepackage{amssymb}
\usepackage{mathtools}
\mathtoolsset{showonlyrefs=true}
\usepackage{mathrsfs}

\usepackage[shortlabels]{enumitem}

\theoremstyle{plain}
\newtheorem{thm}{Theorem}
  \theoremstyle{definition}
  \newtheorem{defn}[thm]{Definition}
  \theoremstyle{definition}
  
  \theoremstyle{remark}
  \newtheorem{rem}[thm]{Remark}
  \theoremstyle{plain}
  \newtheorem{prop}[thm]{Proposition}
  \theoremstyle{plain}
  \newtheorem{lem}[thm]{Lemma}
  \theoremstyle{plain}
  \newtheorem{cor}[thm]{Corollary}
 \theoremstyle{definition}
  \newtheorem{example}[thm]{Example}
  \theoremstyle{remark}
  \newtheorem*{rem*}{Remark}

  \theoremstyle{definition}

\newcommand{\N}{\mathbb{N}}
\newcommand{\R}{{\mathbb{R}}}
\newcommand{\C}{{\mathbb{C}}}
\newcommand{\Z}{{\mathbb{Z}}}
\newcommand{\dd}{{{\rm d}}}
\newcommand{\ii}{{\rm i}}

\newcommand{\spn}{\mathop\mathrm{span}\nolimits}
\newcommand{\Dom}{\mathop\mathrm{Dom}\nolimits}
\newcommand{\der}{\mathop\mathrm{der}\nolimits}
\newcommand{\Ran}{\mathop\mathrm{Ran}\nolimits}
\newcommand{\Ker}{\mathop\mathrm{Ker}\nolimits}
\newcommand{\spec}{\mathop\mathrm{spec}\nolimits}

\newcommand{\dist}{\mathop\mathrm{dist}\nolimits}

\begin{document}

\title[The characteristic function for doubly infinite Jacobi matrices]
{The characteristic function for complex \br doubly infinite Jacobi matrices}
 
\author{Franti\v sek \v Stampach}
\address{Department of Applied Mathematics, 
	Faculty of Information Technology, Czech Technical University in~Prague, 
	Th{\' a}kurova~9, 160~00 Praha, Czech Republic. \br \& \br
	Department of Mathematics, Stockholm University, Kr\"{a}ftriket 5, SE - 106 91 Stockholm, Sweden.}
\email{stampfra@fit.cvut.cz}



\date{\today}

\begin{abstract}
 We introduce a class of doubly infinite complex Jacobi matrices determined by a simple convergence condition
 imposed on the diagonal and off-diagonal sequences. For each Jacobi matrix belonging to this class,
 an analytic function, called a characteristic function, is associated with it. It is shown that the point spectrum
 of the corresponding Jacobi operator restricted to a suitable domain coincides with the zero set of the 
 characteristic function. Also, coincidence regarding the order of a zero of the characteristic function and the
 algebraic multiplicity of the corresponding eigenvalue is proved. Further, formulas for the entries of eigenvectors,
 generalized eigenvectors, a summation identity for eigenvectors, and matrix elements of the resolvent operator 
 are provided. The presented method is illustrated by several concrete examples.
\end{abstract}

\maketitle

\section{Introduction}

In a recent paper \cite{StampachStovicek13}, a method for the spectral analysis of a certain class of semi-infinite
Jacobi matrices based on the so-called characteristic function was developed. The spectral analysis of 
semi-infinite Jacobi matrices is intimately related to classical branches of analysis such as orthogonal
polynomials, moment problems, and continued fractions. There are many monographs that have focused on this, 
see \cite{Akhiezer,Simon_mp,Teschl} for several examples. On the other hand, the spectral theory of doubly infinite 
Jacobi matrices does not appear as often; a nice exposition of this theory can be found in~\cite[Chap.~7]{Berezanskii}.
Some aspects of the spectral theory of the doubly and semi-infinite Jacobi matrices are quite similar (one may
consult, for example, \cite{MassonRepka}); however, some are different. The main difference is the fact that, in the case
of doubly infinite matrices, the space of the solutions to the corresponding second-order difference equation is 
two-dimensional.

The main aim of this article is twofold. First, in analogy with the semi-infinite case treated in~\cite{StampachStovicek13}, 
we introduce the characteristic function associated with the doubly infinite Jacobi matrix
\begin{equation}
  \mathcal{J}=\begin{pmatrix}
	  \ddots & \ddots & \ddots \\
               & w_{-2} & \lambda_{-2} & w_{-1} \\
               & & w_{-1} & \lambda_{-1} & w_{0} \\
               & & & w_{0} & \lambda_{0} & w_{1} \\
               & & & & w_{1} & \lambda_{1} & w_{2} \\
               & & & & & \ddots & \ddots & \ddots 
              \end{pmatrix}\!,
\label{eq:def_J}
\end{equation}
where $\lambda_{n},w_{n}\in\C$ and $w_{n}\neq0$ for all $n\in\Z$. We also show how the characteristic function
can be used to analyze spectral properties of a linear operator acting on $\ell^{2}(\Z)$ whose matrix 
representation with respect to the standard basis of $\ell^{2}(\Z)$ coincides with $\mathcal{J}$. 
Second, we extend the method by proving a result concerning the algebraic multiplicity of eigenvalues and
generalized eigenvectors.

The subclass of matrices $\mathcal{J}$ for which the characteristic function is well defined is determined by
a simple convergence condition imposed on $\lambda_{n}$ and $w_{n}$, see \eqref{eq:assum_sum_w}. The 
self-adjointness of an operator associated with $\mathcal{J}$ is not essential for the presented method; 
therefore, we treat the matrix $\mathcal{J}$ with complex entries, which might be of interest from the point 
of view of the spectral theory of non-self-adjoint operators - a currently very active and rapidly developing field
\cite{Davies,EmbreeTrefethen,Helffer}. The main results of this paper are stated in Theorems~\ref{thm:main} and~\ref{thm:multiplicity}. 

In more detail, we show (under a mild assumption additional to the essential convergence condition \eqref{eq:assum_sum_w})
that the matrix $\mathcal{J}$ uniquely determines a densely defined closed (Jacobi) operator $J$ whose spectral 
properties are related to the properties of the corresponding characteristic function. Namely, by being restricted to a certain 
subset of $\C$, the spectrum of $J$ in this set is discrete and coincides with the set of zeros of the characteristic 
function. Further, we provide formulas for eigenvectors, a summation formula for the entries of an eigenvector and an 
expression for the entries of the matrix representation of the resolvent operator. These results are worked out within Section~\ref{sec:char_func} and 
represent a doubly infinite analog to the corresponding results derived in~\cite{StampachStovicek13} for the case of 
semi-infinite Jacobi matrices.

Section \ref{sec:multiplicity} is devoted to the connection between the order of zeros of the characteristic function
and the algebraic multiplicity of the corresponding Jacobi operator's eigenvalues. In addition, we provide formulas 
for basis vectors of generalized eigenspaces. These results are of particular interest when questions on
diagonalization of non-self-adjoint Jacobi operators are examined.

In Section \ref{sec:diag_reg}, we impose some additional conditions on the diagonal sequence of $\mathcal{J}$, which 
allows us to remove singularities of the characteristic function, with possibly one exception located at 
the origin, and introduce a regularized characteristic function. According to the type of the additional condition, 
we distinguish 3 different cases and illustrate the respective results on concrete examples. Moreover, in 2 cases 
concerning either a compact operator or an operator with compact resolvent, we indicate a connection between the 
regularized characteristic function and the theory of regularized determinants.

\section{The characteristic function and doubly infinite Jacobi matrices}\label{sec:char_func}

In this section, we introduce the characteristic function associated with the doubly infinite Jacobi matrix~\eqref{eq:def_J}
and derive spectral results similar to those obtained in~\cite{StampachStovicek13}. Where the verification of a result
is completely analogous to the corresponding one given in~\cite{StampachStovicek13} for the semi-infinite matrix, 
the proof is only indicated for the sake of brevity.

\subsection{Function $\mathfrak{F}$}

The main algebraic tool for the definition of the characteristic function is a function called $\mathfrak{F}$ 
which appeared in \cite{StampachStovicek11} for the first time. The definition given below is a slight generalization 
of the original one from~\cite[Def.~1]{StampachStovicek11} and is consistent with the one mentioned 
in~\cite[Sec.~2]{StampachStovicek15}.

\begin{defn}
We define $\mathfrak{F}:\Dom\mathfrak{F}\to\C$ by the formula
\begin{equation}
\mathfrak{F}\!\left(\left\{ x_{k}\right\} _{k=-\infty}^{\infty}\right):=1+\sum_{m=1}^{\infty}(-1)^{m}\sum_{k_{1}=-\infty}^{\infty}
\sum_{k_{2}=k_{1}+2}^{\infty}\dots\sum_{k_{m}=k_{m-1}+2}^{\infty}\;\prod_{j=1}^{m}x_{k_{j}}x_{k_{j}+1},
\label{eq:def_F}
\end{equation}
where
\[
\Dom\mathfrak{F}:=\left\{\left\{ x_{k}\right\} _{k=-\infty}^{\infty}\subset\C \;\bigg| \;
\sum_{k=-\infty}^{\infty}\left|x_{k}x_{k+1}\right|<\infty\right\}\!.
\]
Further, if $\left\{ x_{k}\right\} _{k=n_{1}}^{n_{2}}$, with $n_{1},n_{2}\in\mathbb{Z}\cup\{\pm\infty\}$, 
$n_{1}\leq n_{2}$, is given, we put
\[
 \mathfrak{F}\!\left(\left\{ x_{k}\right\}_{k=n_{1}}^{n_{2}}\right):=\mathfrak{F}\!\left(\left\{ x_{k}\right\}_{k=-\infty}^{\infty}\right)\!,
\]
where $x_{k}:=0$, whenever $k<n_{1}$ or $k>n_{2}$, and provided that $\{x_{k}\}_{k=-\infty}^{\infty}\in\Dom\mathfrak{F}$.
Conventionally, for $n_{1},n_{2}\in\mathbb{Z}$, we also put
\[
\mathfrak{F}\!\left(\left\{ x_{k}\right\} _{k=n_{1}}^{n_{2}}\right):=1, \; \mbox{ if } n_{2}=n_{1}-1,\; \mbox{ and } \;
\mathfrak{F}\!\left(\left\{ x_{k}\right\} _{k=n_{1}}^{n_{2}}\right):=0, \; \mbox{ if } n_{2}=n_{1}-2.
\]
\end{defn}

\begin{rem}\label{rem:def_F}$\ $
  Note that the absolute value of the $m$th summand on the RHS of \eqref{eq:def_F} is majorized by the expression
 \[
  \sum_{\substack{k\in\mathbb{Z}^{m}\\k_{1}<k_{2}<\dots<k_{m}}}\prod_{j=1}^{m}|x_{k_{j}}x_{k_{j}+1}|
  \leq\frac{1}{m!}\left(\sum_{k=-\infty}^{\infty}|x_{k}x_{k+1}|\right)^{\! m},
\]cf. \cite[Rem.~2]{StampachStovicek11}.
  Consequently, the function $\mathfrak{F}$ is well defined on sequences from $\Dom\mathfrak{F}$, and we have the estimate
  \[
   \left|\mathfrak{F}\!\left(\left\{ x_{k}\right\}_{k=-\infty}^{\infty}\right)\right|\leq\exp\left(\sum_{k=-\infty}^{\infty}\left|x_{k}x_{k+1}\right|\right)\!.
  \]
\end{rem}

Recall several properties of $\mathfrak{F}$. In the following formulas we always assume that all the expressions are well defined,
i.e., the vector in the argument of $\mathfrak{F}$ (possibly with additional zeros) belongs to $\Dom\mathfrak{F}$. 
First of all, we have the important relation \cite[Eq.~(19)]{StampachStovicek13}
\begin{align}
\mathfrak{F}\!\left(\left\{ x_{k}\right\}_{k=n_{1}}^{n_{2}}\right)=&\,\mathfrak{F}\!\left(\left\{ x_{k}\right\}_{k=n_{1}}^{n}\right)\mathfrak{F}\!\left(\left\{ x_{k}\right\}_{k=n+1}^{n_{2}}\right)\nonumber\\
&-x_{n}x_{n+1}\,\mathfrak{F}\!\big(\left\{x_{k}\right\}_{k=n_{1}}^{n-1}\big)\mathfrak{F}\!\left(\left\{x_{k}\right\}_{k=n+2}^{n_{2}}\right)\!,
\label{eq:F_recur_general}
\end{align}
where $n\in\mathbb{Z}$ satisfies $n_{1}\leq n \leq n_{2}$ and $n_{1}$, $n_{2}\in\Z\cup\{\pm\infty\}$. Equivalently, \eqref{eq:F_recur_general} can be written as
\begin{align}
 \mathfrak{F}\!\left(\left\{ x_{k}\right\}_{k=n_{1}}^{n_{2}}\right)=&\,\mathfrak{F}\!\left(\left\{ x_{k}\right\}_{k=n_{1}}^{n}\right)\mathfrak{F}\!\left(\left\{ x_{k}\right\} _{k=n+1}^{n_{2}}\right)+
 \mathfrak{F}\!\big(\left\{ x_{k}\right\} _{k=n_{1}}^{n-1}\big)\mathfrak{F}\!\big(\left\{ x_{k}\right\} _{k=n}^{n_{2}}\big)\nonumber\\
 &-\mathfrak{F}\!\big(\left\{ x_{k}\right\}_{k=n_{1}}^{n-1}\big)\mathfrak{F}\!\left(\left\{ x_{k}\right\} _{k=n+1}^{n_{2}}\right)\!, \label{eq:F_recur_general_3term}
\end{align}
where we have used that
\begin{equation}
  \mathfrak{F}\!\left(\left\{ x_{k}\right\} _{k=n}^{n_{2}}\right)=\mathfrak{F}\!\left(\left\{ x_{k}\right\} _{k=n+1}^{n_{2}}\right)
 -x_{n}x_{n+1}\,\mathfrak{F}\!\left(\left\{ x_{k}\right\} _{k=n+2}^{n_{2}}\right)\!,\quad n\leq n_{2},
 \label{eq:F_recur_+inf}
\end{equation}
which is the special case of \eqref{eq:F_recur_general} with $n_{1}:=n$. Similarly, by putting $n_{2}:=n+1$ in \eqref{eq:F_recur_general}, we obtain
\begin{equation}
  \mathfrak{F}\!\big(\left\{ x_{k}\right\} _{k=n_{1}}^{n+1}\big)=\mathfrak{F}\!\left(\left\{ x_{k}\right\} _{k=n_{1}}^{n}\right)
 -x_{n}x_{n+1}\,\mathfrak{F}\!\big(\left\{ x_{k}\right\}_{k=n_{1}}^{n-1}\big), \quad n\geq n_{1}-1.
 \label{eq:F_recur_-inf}
\end{equation}
Second, one has the limit relations
\begin{equation}
 \lim_{n\to\infty}\mathfrak{F}\!\left(\left\{ x_{k}\right\}_{k=n}^{\infty}\right)=1 \; \mbox{ and } \; \lim_{n\to-\infty}\mathfrak{F}\!\left(\left\{ x_{k}\right\}_{k=-\infty}^{n}\right)=1
\label{eq:limits_F_1}
\end{equation}
and
\[
\mathfrak{F}\!\left(\left\{ x_{k}\right\}_{k=-\infty}^{\infty}\right)=\lim_{n\to-\infty}\mathfrak{F}\!\left(\left\{ x_{k}\right\}_{k=n}^{\infty}\right)=
\lim_{n\to\infty}\mathfrak{F}\!\left(\left\{ x_{k}\right\}_{k=-\infty}^{n}\right)\!,
\]
which one verifies in the same way as \cite[Lem.~2]{StampachStovicek13}.

\subsection{The characteristic function}

For two given sequences satisfying a certain convergence condition, we will define a complex function defined on a subset of $\C$ in terms of $\mathfrak{F}$.
This function will be called the characteristic function associated with the doubly infinite Jacobi matrix $\mathcal{J}$ since it plays a similar role to the characteristic 
polynomial in linear algebra, as it will be further demonstrated.

First, let us introduce a notation we will use. For $\lambda:\Z\to\C$ a complex sequence, we put
$\Ran(\lambda):=\{\lambda_{n}\mid n\in\Z\}$ and $\C_{0}^{\lambda}:=\C\setminus\overline{\Ran(\lambda)}$, where
$\overline{\Ran(\lambda)}$ is the closure of $\Ran(\lambda)$. Further, we denote by $\der(\lambda)$ the set of all
(finite) accumulation points of $\Ran(\lambda)$, i.e., $\der(\lambda)$ is the set of all limit points of all possible convergent subsequences of $\lambda$.
Clearly, $\overline{\Ran(\lambda)}=\Ran(\lambda)\cup\der(\lambda)$.

For $\lambda,w:\Z\to\C$, the characteristic function will be defined under the condition
\begin{equation}
 \sum_{n=-\infty}^{\infty}\left|\frac{w_{n}^{2}}{(\lambda_{n}-z_{0})(\lambda_{n+1}-z_{0})}\right|<\infty
\label{eq:assum_sum_w}
\end{equation}
that is valid for at least one $z_{0}\in\C_{0}^{\lambda}$. If this is true, then \eqref{eq:assum_sum_w} remains valid for all $z_{0}\in\C_{0}^{\lambda}$ and
the convergence is locally uniform on $\mathbb{C}_{0}^{\lambda}$, as it is straightforward to verify; cf. the proof of 
\cite[Lem.~8]{StampachStovicek13}.
In fact, one also readily shows that \eqref{eq:assum_sum_w} remains valid also for $z\in\Ran(\lambda)\setminus\der(\lambda)$ provided the
finite number of terms where we would divide by zero are omitted in the series.

\begin{defn}\label{def:char_fction}
 Let $\lambda:\Z\to\C$ and $w:\Z\to\C\setminus\{0\}$ be such that \eqref{eq:assum_sum_w} holds for at least one $z_{0}\in\C_{0}^{\lambda}$. We define the \emph{characteristic
 function associated with the doubly infinite matrix} $\mathcal{J}$ given by \eqref{eq:def_J} by
 \begin{equation}
  F_{\mathcal{J}}(z):=\mathfrak{F}\!\left(\left\{\frac{\gamma_{n}^{2}}{z-\lambda_{n}}\right\}_{n=-\infty}^{\infty}\right)\!, \quad \forall z\in\C_{0}^{\lambda},
 \label{eq:def_char_fction}
 \end{equation}
 where $\gamma:\Z\to\C$ is any sequence satisfying the difference equation $\gamma_{n}\gamma_{n+1}=w_{n}$ for all $n\in\Z$.
\end{defn}

\begin{rem}
 Note that since the sequence in the argument of $\mathfrak{F}$ in \eqref{eq:def_char_fction} fulfills \eqref{eq:assum_sum_w} for all $z_0\in\C_{0}^{\lambda}$,
 it belongs to $\Dom\mathfrak{F}$ and the RHS of \eqref{eq:def_char_fction} is therefore well defined for all $z\in\C_{0}^{\lambda}$. 
 Further, since $0\notin\Ran w$, the sequence $\gamma$ always exists and is determined uniquely by specifying one value, for example, by setting $\gamma_{0}=1$.
 The definition of the characteristic function does not depend on a particular choice of the sequence $\gamma$.
\end{rem}

By a simple modification of the argument used in the proof of~\cite[Lem.~8]{StampachStovicek13}, one verifies that
\begin{equation}
 \lim_{n\rightarrow\infty}\mathfrak{F\!}\left(\left\{ \frac{\gamma_{k}^{\,2}}{\lambda_{k}-z}\right\} _{k=-n}^{n}\right)=F_{\mathcal{J}}(z),
 \label{eq:lim_calF_FJ}
\end{equation}
and the convergence is locally uniform on $\mathbb{C}_{0}^{\lambda}$, provided the condition \eqref{eq:assum_sum_w} holds for at least one $z_{0}\in\C_{0}^{\lambda}$.
Consequently, $F_{\mathcal{J}}$ is an analytic function on $\mathbb{C}_{0}^{\lambda}$. Moreover, $F_{\mathcal{J}}$ is meromorphic on $\C\setminus\der(\lambda)$. Indeed,
the singularity of $F_{\mathcal{J}}$ at a point $z=\lambda_{n}$, for some $n\in\Z$ such that $\lambda_{n}\not\in\der(\lambda)$, is either a removable singularity
or a pole of order less than or equal to $r(z)$ where
\begin{equation}
 r(z):=\#\{n\in\Z \mid \lambda_{n}=z\}
\label{eq:def_r}
\end{equation}
is the number of values of $\lambda$ coinciding with $z$. To see this, for $z\in\Ran(\lambda)\setminus\der(\lambda)$ fixed, take
$m,M\in\Z$, $m\leq M$, such that $\lambda_{n}\neq z$ for all $n\leq m$ and all $n\geq M$. Using the rule \eqref{eq:F_recur_general_3term}
twice, one derives, for $u\in\C_{0}^{\lambda}$, that
\begin{align}
&F_{\mathcal{J}}(u)=\left(F_{-\infty}^{m}(u)-F_{-\infty}^{m-1}(u)\right)\big[F_{m+1}^{M}(u)F_{M+1}^{\infty}(u)+F_{m+1}^{M-1}(u)F_{M}^{\infty}(u)\nonumber\\
&\hskip232pt-F_{m+1}^{M-1}(u)F_{M+1}^{\infty}(u)\big]\nonumber\\
&\hskip23pt+F_{-\infty}^{m-1}(u)\left[F_{m}^{M}(u)F_{M+1}^{\infty}(u)+F_{m}^{M-1}(u)F_{M}^{\infty}(u)-F_{m}^{M-1}(u)F_{M+1}^{\infty}(u)\right]\!, \label{eq:order_pole_eq}
\end{align}
where we temporarily denote
\[
 F_{r}^{s}(u):=\mathfrak{F}\!\left(\left\{\frac{\gamma_{n}^{2}}{u-\lambda_{n}}\right\}_{n=r}^{s}\right)\!, \quad r,s\in\Z\cup\{\pm\infty\}, r\leq s,
\]
for brevity. Clearly, only functions $F_{m+1}^{M}$, $F_{m+1}^{M-1}$, $F_{m}^{M}$, and $F_{m}^{M-1}$ can have a singularity at $z$ which can be either a removable singularity or a pole of order at most $r(z)$.
The remaining terms on the RHS of \eqref{eq:order_pole_eq} are functions analytic at~$z$.

\subsection{The Jacobi operator}

Let us recall the standard procedure of prescribing densely defined and closed operators associated with $\mathcal{J}$. 
On the algebraic level, the formal doubly infinite matrix $\mathcal{J}$ can be understood as a linear
mapping acting on the space of complex sequences $x$ (indexed by $\Z$) by a formal matrix multiplication, i.e.,
\begin{equation}
 (\mathcal{J}x)_{n}=w_{n-1}x_{n-1}+\lambda_{n}x_{n}+w_{n}x_{n+1}, \quad \forall n\in\Z.
\label{eq:Jx_n_eq}
\end{equation}

Let $\{e_{n} \mid n\in\Z\}$ stands for the standard basis of $\ell^{2}(\Z)$, i.e., $(e_{n})_{m}=\delta_{m,n}$ for $m,n\in\Z$. Define an auxiliary operator $J_{0}$ as
\[
 J_{0}:=\mathcal{J}\upharpoonleft\spn\{e_{n}\mid n\in\Z\}.
\]
The operator $J_{0}$, as an operator on the Hilbert space $\ell^{2}(\Z)$, need not be closed in $\ell^{2}(\Z)$, 
but is always closable, cf. \cite[Subsec.~2.1]{Beckermann01}.
Hence, one can introduce the so-called minimal operator $J_{\min}$ as the operator closure of~$J_{0}$.
On the other hand, it is natural to define the maximal operator $J_{\max}$ by putting
\[
 \Dom J_{\max}:=\{x\in\ell^{2}(\Z) \mid \mathcal{J}x\in\ell^{2}(\Z)\}
\]
and
\[
  J_{\max}x:=\mathcal{J}x, \; \mbox{ for } x\in\Dom J_{\max}.
\]
Here, the expression $\mathcal{J}x$ is to be understood as in \eqref{eq:Jx_n_eq}.

One can show that $J_{\min}\subset J_{\max}$, but the equality does not hold in general. The operators $J_{\min}$ and 
$J_{\max}$ are related via their adjoints:
\begin{equation}
 J_{\max}=\mathcal{C}J_{\min}^{*}\mathcal{C} \;\mbox{ and }\; J_{\min}=\mathcal{C}J_{\max}^{*}\mathcal{C},
\label{eq:J_max_rel_J_min}
\end{equation}
where $\mathcal{C}$ stands for the complex conjugation operator acting on complex sequences as $(\mathcal{C}x)_{n}:=\overline{x_{n}}$, $n\in\Z$.
The verification of formulas in \eqref{eq:J_max_rel_J_min} is a matter of straightforward use of the definition of the adjoint operator; see also
\cite[Lem.~2.1]{Beckermann01} for an analogous proof for operators associated with a semi-infinite complex Jacobi matrix. 

It follows from \eqref{eq:J_max_rel_J_min} that $J_{\max}$ is closed. In addition, any closed linear operator $A$ acting on $\ell^{2}(\Z)$ such that
$\{e_n \mid n\in\Z\}\subset\Dom A$, whose matrix representation with respect to the standard basis coincides with $\mathcal{J}$, satisfies
$J_{\min}\subset A\subset J_{\max}$. If $J_{\min}=J_{\max}$, the Jacobi matrix $\mathcal{J}$ determines uniquely the Jacobi operator.
In this case, the subscripts can be omitted and we simply write $J:=J_{\min}=J_{\max}$.

\subsection{Spectral properties of the Jacobi operator via the characteristic function}

For $\lambda:\Z\to\C$, $w:\Z\to\C\setminus\{0\}$, and $z\notin\Ran(\lambda)$, we put
\begin{equation}
 \mathcal{P}_{n}(z):=\prod_{k=1}^{n}\frac{w_{k-1}}{z-\lambda_{k}}, \; \mbox{ if } n\geq0, \; \mbox{ and } \;
 \mathcal{P}_{n}(z):=\prod_{k=n+1}^{0}\frac{z-\lambda_{k}}{w_{k-1}}, \; \mbox{ if } n<0.
\label{eq:def_P_z}
\end{equation}
Note that $\mathcal{P}(z):\Z\to\C\setminus\{0\}$ satisfies the equations
 \begin{equation}
  \mathcal{P}_{0}(z)=1, \; \mbox{ and } \; \mathcal{P}_{n+1}(z)=\frac{w_{n}}{z-\lambda_{n+1}}\mathcal{P}_{n}(z), \quad \forall n\in\Z.
 \label{eq:P_z_recur}
 \end{equation}
Further, for $\lambda:\Z\to\C$ and $z\in\C$, we introduce quantities
\begin{equation}
   r_{+}(z):=\#\{n>0 \mid \lambda_{n}=z\} \; \mbox{ and } \; r_{-}(z):=\#\{n\leq 0 \mid \lambda_{n}=z\}.
\label{eq:def_r_plus_minus}
\end{equation}
Thus, $r_{\pm}(z)\in\{0,1,2,\dots,\infty\}$ and $r_{+}(z)+r_{-}(z)=r(z)$ where $r(z)$ is defined in \eqref{eq:def_r}. 
Note that $r_{\pm}(z)<\infty$, if $z\notin\der(\lambda)$, and $r_{\pm}(z)=0$, if $z\notin\Ran(\lambda)$. 

\begin{defn}\label{def:sol_f_g}
 Let $\lambda:\Z\to\C$ and $w:\Z\to\C\setminus\{0\}$ be such that \eqref{eq:assum_sum_w} holds for at least one $z_{0}\in\C_{0}^{\lambda}$. For $z\in\C_{0}^{\lambda}$, we define
 two sequences $f(z),g(z):\Z\to\C$ by putting
 \begin{equation}
  f_{n}(z):=\mathcal{P}_{n}(z)\mathfrak{F\!}\left(\left\{ \frac{\gamma_{k}^{\,2}}{\lambda_{k}-z}\right\} _{k=n+1}^{\infty}\right)\!,
 \label{eq:def_f_z}
 \end{equation}
 and
 \begin{equation}
  g_{n}(z):=\frac{1}{w_{n-1}\mathcal{P}_{n-1}(z)}\mathfrak{F\!}\left(\left\{ \frac{\gamma_{k}^{\,2}}{\lambda_{k}-z}\right\}_{k=-\infty}^{n-1}\right)\!,
 \label{eq:def_g_z}
 \end{equation}
 for $n\in\Z$, where $\gamma:\Z\to\C$ is as in Definition \ref{def:char_fction} and $\mathcal{P}_{n}(z)$ given by \eqref{eq:def_P_z}. Further, we extend the definition of $f(z)$ and $g(z)$ for $z\in\Ran(\lambda)\setminus\der(\lambda)$ by formulas
 \begin{equation}
  f_{n}(z):=\lim_{u\to z}(u-z)^{r_{+}(z)}f_{n}(u) \quad \mbox{ and } \quad g_{n}(z):=\lim_{u\to z}(u-z)^{r_{-}(z)}g_{n}(u),
 \label{eq:def_f_g_lim_ext}
 \end{equation}
 where $r_{\pm}(z)$ are defined by \eqref{eq:def_r_plus_minus}.
 \end{defn}
 
\begin{rem}
Note that the sequences $f(z)$ and $g(z)$ are well defined by Definition \ref{def:sol_f_g} for all $z\in\C\setminus\der(\lambda)$. In addition, for $z\in\C\setminus\der(\lambda)$
and $n>0$ such that $n\geq\max\{k\in\Z \mid \lambda_{k}=z\}$, one has
 \begin{equation}
  f_{n}(z)=\left(\prod_{\substack{k=1 \\ \lambda_{k}\neq z}}^{n}\frac{w_{k-1}}{z-\lambda_{k}}\right)\mathfrak{F\!}\left(\left\{ \frac{\gamma_{k}^{\,2}}{\lambda_{k}-z}\right\} _{k=n+1}^{\infty}\right)\!.
 \label{eq:f_n_n_large}
 \end{equation}
Similarly, for $n\leq0$ such that $n\leq\min\{k\in\Z \mid \lambda_{k}=z\}$, one has
 \begin{equation}
  g_{n}(z)=\frac{1}{w_{n-1}}\left(\prod_{\substack{k=n \\ \lambda_{k}\neq z}}^{0}\frac{w_{k-1}}{z-\lambda_{k}}\right)\mathfrak{F\!}\left(\left\{ \frac{\gamma_{k}^{\,2}}{\lambda_{k}-z}\right\} _{k=-\infty}^{n-1}\right)\!.
 \label{eq:g_n_n_small}
 \end{equation}
\end{rem}

\begin{prop}\label{prop:sol_f_g}
  Let $\lambda:\Z\to\C$ and $w:\Z\to\C\setminus\{0\}$ be such that \eqref{eq:assum_sum_w} holds for at least one $z_{0}\in\C_{0}^{\lambda}$. Then $f(z)$ and $g(z)$
  are solutions of the eigenvalue equation $\mathcal{J}u=zu$ for all $z\in\C\setminus\der(\lambda)$. In addition, for 
  their Wronskian $W(f(z),g(z)):=w_{n}\left(f_{n}(z)g_{n+1}(z)
  -f_{n+1}(z)g_{n}(z)\right)$ ($n$ is arbitrary), one has
  \[
   W(f(z),g(z))=\lim_{u\to z}(u-z)^{r(z)}F_{\mathcal{J}}(u), \quad \forall z\in\C\setminus\der(\lambda).
  \]
\end{prop}

\begin{proof}
 We verify the statement for $z\in\C_{0}^{\lambda}$ only. The extension to all $z\in\C\setminus\der(\lambda)$ 
 is to be treated readily with the aid of limit formulas~\eqref{eq:def_f_g_lim_ext}.
 
By using the definition relations \eqref{eq:def_f_z}, \eqref{eq:def_g_z} and taking into account the recurrence \eqref{eq:P_z_recur}, the verification of the equation $\mathcal{J}u=zu$, i.e., 
 \[
  w_{n-1}u_{n-1}+(\lambda_{n}-z)u_{n}+w_{n}u_{n+1}=0, \quad \forall n\in\Z,
 \]
 for $u=f(z)$ and $u=g(z)$, is a straightforward application of the identities \eqref{eq:F_recur_+inf} and \eqref{eq:F_recur_-inf}.
 
 Further, for $n\in\Z$ and $z\in\C_{0}^{\lambda}$ arbitrary, we have
 \begin{align*}
  W&(f(z),g(z))=\mathfrak{F\!}\left(\left\{ \frac{\gamma_{k}^{\,2}}{\lambda_{k}-z}\right\}_{k=-\infty}^{n}\right)\mathfrak{F\!}\left(\left\{ \frac{\gamma_{k}^{\,2}}{\lambda_{k}-z}\right\}_{k=n+1}^{\infty}\right)\\
  &\hskip18pt-\frac{w_{n}^{2}}{(z-\lambda_{n})(z-\lambda_{n+1})}\mathfrak{F\!}\left(\left\{ \frac{\gamma_{k}^{\,2}}{\lambda_{k}-z}\right\}_{k=-\infty}^{n-1}\right)\mathfrak{F\!}\left(\left\{ \frac{\gamma_{k}^{\,2}}{\lambda_{k}-z}\right\} _{k=n+2}^{\infty}\right)\!.
 \end{align*}
 According to \eqref{eq:F_recur_general}, the RHS of the above equality coincides with $F_{\mathcal{J}}(z)$ which completes the proof.
\end{proof}

For the spectral analysis of a Jacobi operator associated to $\mathcal{J}$, we will need to know the asymptotic behavior
of the sequences $f(z)$ and $g(z)$ from Definition~\ref{def:sol_f_g}, for $z\in\C\setminus\der(\lambda)$, as the index 
approaches $+\infty$ and $-\infty$, respectively. The following auxiliary result will be used for this purpose.

\begin{lem}\label{lem:sum_P_at_inf}
 Let $\lambda:\Z\to\C$ and $w:\Z\to\C\setminus\{0\}$ be such that \eqref{eq:assum_sum_w} holds for at least one $z_{0}\in\C_{0}^{\lambda}$.
 Then
 \[
  \sum_{n=1}^{\infty}\prod_{\substack{k=1 \\ \lambda_{k}\neq z}}^{n}\left|\frac{w_{k-1}}{z-\lambda_{k}}\right|<\infty \quad \mbox{ and } 
  \quad \sum_{n=-\infty}^{0}\frac{1}{|w_{n-1}|}\prod_{\substack{k=n \\ \lambda_{k}\neq z}}^{0}\left|\frac{w_{k-1}}{z-\lambda_{k}}\right|<\infty,
 \]
 for all $z\in\C\setminus\der(\lambda)$.
\end{lem}

\begin{proof}
 We verify the convergence of the first series. The second one is to be treated similarly.
 
 Let $z\in\C\setminus\der(\lambda)$ be fixed and denote by $M(z)\in\N$ an index such that $z\neq \lambda_{n}$ 
 for all $n\geq M(z)$. Since
 \[
  \sum_{n=M(z)}^{\infty}\left|\frac{w_{n}^{2}}{(\lambda_{n}-z)(\lambda_{n+1}-z)}\right|<\infty
 \]
 one has
 \[
  \lim_{n\to\infty}\frac{w_{n}^{2}}{(\lambda_{n}-z)(\lambda_{n+1}-z)}=0
 \]
 and hence, without loss of generality, we may assume that
 \begin{equation}
  |w_{n}|\leq\frac{1}{2}\left|(\lambda_{n}-z)(\lambda_{n+1}-z)\right|^{1/2}, \quad \forall n\geq M(z).
 \label{eq:w_n_bound_lem_in_proof}
 \end{equation}
 Next, with the aid of~\eqref{eq:w_n_bound_lem_in_proof}, one obtains
 \begin{equation}
  \prod_{\substack{k=1 \\ \lambda_{k}\neq z}}^{n}\left|\frac{w_{k-1}}{z-\lambda_{k}}\right|\leq \frac{C(z)}{2^{n}|z-\lambda_{n}|^{1/2}}, \quad \forall n\geq M(z),
 \label{eq:estim_conv_major_lem_in_proof}
 \end{equation}
 where
 \[
  C(z):=2^{M(z)}|z-\lambda_{M(z)}|^{1/2}\prod_{\substack{k=1 \\ \lambda_{k}\neq z}}^{M(z)}\left|\frac{w_{k-1}}{z-\lambda_{k}}\right|\!.
 \]
 Since $z\neq\lambda_{n}$ for all $n\geq M(z)$ and $z\notin\der(\lambda)$,
 \[
 |z-\lambda_{n}|\geq\dist\left(z,\overline{\{\lambda_{m}\mid m\geq M(z)\}}\right)>0, \quad \forall n\geq M(z),
 \]
 hence, the RHS of the inequality \eqref{eq:estim_conv_major_lem_in_proof} is the convergent majorant for the first series
 from the statement. 
\end{proof}

 Next, for the purpose of the main theorem of this section, we introduce an extended zero set of the characteristic function $F_{\mathcal{J}}$.

\begin{defn}
For $\lambda:\Z\to\C$ and $w:\Z\to\C\setminus\{0\}$ being such that \eqref{eq:assum_sum_w} holds for some $z_{0}\in\C_{0}^{\lambda}$, we define
\[
 \mathfrak{Z}(\mathcal{J}):=\left\{z\in\C\setminus\der(\lambda) \;\bigg|\; \lim_{u\to z}(u-z)^{r(z)}F_{\mathcal{J}}(z)=0\right\}\!.
\]
\end{defn}

\begin{rem}
Note that $\mathfrak{Z}(\mathcal{J})$ decomposes into the union of the set of all zeros of $F_{\mathcal{J}}$ located in $\C_{0}^{\lambda}$ (since $r(z)=0$ for all
$z\in\C_{0}^{\lambda}$) and the set of those points from $\Ran(\lambda)\setminus\der(\lambda)$ which are not the poles of $F_{\mathcal{J}}$ of order $r(z)$, i.e., 
they are either removable singularities or poles of order strictly less then $r(z)$.
\end{rem}

\begin{prop}\label{prop:Z_in_spec_p_J_max}
 Let $\lambda:\Z\to\C$ and $w:\Z\to\C\setminus\{0\}$ be such that \eqref{eq:assum_sum_w} holds for at least one $z_{0}\in\C_{0}^{\lambda}$. Then one has
 \begin{equation}
  \mathfrak{Z}(\mathcal{J})\subset\spec_{p}(J_{\max})\setminus\der(\lambda)
 \label{eq:Z_in_spec_p_J_max}
 \end{equation}
 and, for $z\in\mathfrak{Z}(\mathcal{J})$, the corresponding eigenvector of $J_{\max}$ can be chosen as $f(z)$.
\end{prop}

\begin{proof} 
For $z\in\C\setminus\der(\lambda)$ fixed, we have, by \eqref{eq:limits_F_1}, 
\eqref{eq:f_n_n_large}, and \eqref{eq:g_n_n_small}, that
 \begin{equation}
  f_{n}(z)=\alpha_{n}(z)\left[1+o(1)\right], \quad \mbox{ as } n\to\infty,
 \label{eq:f_asympt}
 \end{equation}
 and
 \[
  g_{n}(z)=\left(w_{n-1}\alpha_{n-1}(z)\right)^{-1}\left[1+o(1)\right], \quad \mbox{ as } n\to-\infty,
 \]
 where
 \[
  \alpha_{n}(z)=\prod_{\substack{k=1 \\ \lambda_{k}\neq z}}^{n}\frac{w_{k-1}}{z-\lambda_{k}}, \; \mbox{ for } n>0, \; \mbox{ and } \;
  \alpha_{n}(z)=\prod_{\substack{k=n \\ \lambda_{k}\neq z}}^{0}\frac{z-\lambda_{k}}{w_{k-1}}, \; \mbox{ for } n\leq 0.
 \]
 Lemma \ref{lem:sum_P_at_inf} implies that $\alpha(z)$ is a summable sequence at $+\infty$ and 
 $(w\alpha(z))^{-1}$ is a summable sequence at $-\infty$. Consequently, for all $z\in\C\setminus\der(\lambda)$,
 $f(z)$ is square summable at $+\infty$ and $g(z)$ square summable at $-\infty$.
 
 Assume $z\in\mathfrak{Z}(\mathcal{J})$. Then, according to Proposition \ref{prop:sol_f_g}, $f(z)$ and $g(z)$ are two  solutions of the eigenvalue equation $\mathcal{J}u=zu$, which are linearly dependent 
 since their Wronskian vanishes. Hence, by the above discussion, $f(z)$ and $g(z)$ belong to $\ell^{2}(\Z)$. 
 Particularly, we have $f(z)\in\ell^{2}(\Z)$ and $J_{\max}f(z)=zf(z)$. Thus, if $f(z)\neq0$, then $z$ is an eigenvalue of $J_{\max}$ and $f(z)$ the corresponding eigenvector.
 Assume, on the contrary, that $f(z)=0$. Then, by~\eqref{eq:f_n_n_large}, one has
 \[
  \mathfrak{F\!}\left(\left\{ \frac{\gamma_{k}^{\,2}}{\lambda_{k}-z}\right\} _{k=n}^{\infty}\right)=0,
 \]
 for all $n\in\N$ sufficiently large. However, according to \eqref{eq:limits_F_1}, the LHS of the above equality tends to $1$, as $n\to\infty$, which is a contradiction.
 \end{proof}
 
Next, we derive a summation formula which, for real $w$ and $\lambda$, turns out to be the $\ell^{2}$-norm of 
an eigenvector of $J_{\max}$ corresponding to an eigenvalue located in $\C_{0}^{\lambda}$.

\begin{prop}\label{prop:l2_norm}
  Let $\lambda:\Z\to\C$ and $w:\Z\to\C\setminus\{0\}$ be such that \eqref{eq:assum_sum_w} holds for some $z_{0}\in\C_{0}^{\lambda}$
  and let $z\in\C_{0}^{\lambda}$ be a zero of $F_{\mathcal{J}}$. Then one has
 \begin{equation}
  \sum_{n=-\infty}^{\infty}f_{n}^{2}(z)=A(z)F_{\mathcal{J}}'(z),
  \label{eq:sum_f_squared}
 \end{equation}
 where $A(z)$ is given by the formula
  \begin{equation}
   A(z)=w_{n-1}\mathcal{P}_{n}(z)\mathcal{P}_{n-1}(z)
  \mathfrak{F\!}\left(\left\{ \frac{\gamma_{k}^{\,2}}{\lambda_{k}-z}\right\} _{k=n+1}^{\infty}\right)\Bigg/
  \mathfrak{F\!}\left(\left\{ \frac{\gamma_{k}^{\,2}}{\lambda_{k}-z}\right\} _{k=-\infty}^{n-1}\right)\!,
  \label{eq:A_z_formula}
  \end{equation}
  with an arbitrary $n\in\Z$ such that the denominator does not vanish.
 \end{prop}
 
 \begin{proof}
  According to Proposition \ref{prop:sol_f_g}, $f(z)$ is a solution of the second-order difference equation
  \[
   w_{n-1}u_{n}+(\lambda_{n}-z)u_{n}+w_{n}u_{n+1}=0, \quad \forall n\in\Z, \forall z\in\C_{0}^{\lambda}.
  \]
  By application of the Green formula, one obtains
  \begin{equation}
   (x-y)\sum_{k=m+1}^{n}f_{k}(x)f_{k}(y)=W_{m}\left(f(x),f(y)\right)-W_{n}\left(f(x),f(y)\right)\!,
   \label{eq:Green_id_f}
  \end{equation}
  for all $m,n\in\Z$, $m\leq n$ and all $x,y\in \C_{0}^{\lambda}$, where $W_{n}(f(x),f(y)):=w_{n}(f_{n}(x)f_{n+1}(y)-f_{n+1}(x)f_{n}(y))$.
  With the aid of the asymptotic formula~\eqref{eq:f_asympt}, one verifies that
  \begin{equation}
   w_{n}f_{n}(x)f_{n+1}(y)=(x-\lambda_{0})\left(\prod_{k=0}^{n}\frac{w_{k}^{2}}{(x-\lambda_{k})(y-\lambda_{k+1})}\right)\left(1+o(1)\right),
  \label{eq:part_Wronsk_f_in_proof}  
  \end{equation}
  for $x,y\in\C_{0}^{\lambda}$ and $n\to\infty$. Note that
  \[
   \left|\frac{w_{n}^{2}}{(x-\lambda_{n})(y-\lambda_{n+1})}\right|\leq\left|\frac{w_{n}^{2}}{(x-\lambda_{n})(x-\lambda_{n+1})}\right|\left(1+\frac{|x-y|}{\dist(y,\Ran(\lambda))}\right)\!.
  \]
  The RHS in the above estimate tends to $0$, as $n\to\infty$, which follows from the assumption~\eqref{eq:assum_sum_w}. Consequently, the product on the RHS of \eqref{eq:part_Wronsk_f_in_proof}
  tends to~$0$, as $n\to\infty$, and hence
  \[
   \lim_{n\to\infty}W_{n}\left(f(x),f(y)\right)=0,
  \]
  for all $x,y\in\C_{0}^{\lambda}$. Thus, by sending $n\to\infty$ in \eqref{eq:Green_id_f}, one gets
  \begin{equation}
   (x-y)\sum_{k=m+1}^{\infty}f_{k}(x)f_{k}(y)=W_{m}\left(f(x),f(y)\right), \quad \forall x,y\in \C_{0}^{\lambda}.
  \label{eq:part_Wronsk_f_in_proof1}
  \end{equation}
  Finally, with the aid of~\eqref{eq:f_asympt} and similarly to what occurs in the proof of Lemma~\ref{lem:sum_P_at_inf}, 
  one verifies the sum on the LHS of~\eqref{eq:part_Wronsk_f_in_proof1} converges locally uniformly in $y$ 
  on $\C_{0}^{\lambda}$ with $x\in\C_{0}^{\lambda}$ fixed (we omit details). Thus, by sending $y\to x$ in~\eqref{eq:part_Wronsk_f_in_proof1}
  we may interchange the limit and the summation getting
  \begin{equation}
   \sum_{k=m+1}^{\infty}f_{k}^{2}(x)=W_{m}\left(f'(x),f(x)\right)\!, \quad \forall x\in \C_{0}^{\lambda}.
  \label{eq:part_sum_f_squared}
  \end{equation}
  Analogously, one proves that
  \begin{equation}
   \sum_{k=-\infty}^{n}g_{k}^{2}(x)=W_{n}\left(g(x),g'(x)\right)\!, \quad \forall x\in \C_{0}^{\lambda}.
  \label{eq:part_sum_g_squared}
  \end{equation}
  
  For $z\in\C_{0}^{\lambda}$ the zero of $F_{\mathcal{J}}(z)$, the sequences $f(z)$ and $g(z)$ are linearly dependent 
  by Proposition \ref{prop:sol_f_g}. Hence there exists $A(z)\neq0$ such that 
  \begin{equation}
  f(z)=A(z)g(z). 
  \label{eq:f_eq_A_g}
  \end{equation}
  Since $f(z)$ is an eigenvector of $J_{\max}$ by Proposition \ref{prop:Z_in_spec_p_J_max}, $A(z)\neq0$, indeed.
  By differentiating both sides of the equality $F_{\mathcal{J}}(x)=W(f(x),g(x))$ and making use of \eqref{eq:f_eq_A_g}, one obtains
  \[
   F_{\mathcal{J}}'(z)=A(z)^{-1}W_{n}\left(f'(z),f(z)\right)+A(z)W_{n}\left(g(z),g'(z)\right), \quad \forall n\in\Z.
  \]
  Further, by substituting from \eqref{eq:part_sum_f_squared} and \eqref{eq:part_sum_g_squared} in the above equality, 
  one gets
  \[
   F_{\mathcal{J}}'(z)=A(z)^{-1}\sum_{k=n+1}^{\infty}f_{k}^{2}(z)+A(z)\sum_{k=-\infty}^{n}g_{k}^{2}(z), \quad \forall n\in\Z.
  \]
  Finally, by sending $n\to-\infty$ in the above formula, one arrives at \eqref{eq:sum_f_squared}.
  To obtain the expression \eqref{eq:A_z_formula} for $A(z)$, it suffices to use~\eqref{eq:f_eq_A_g} and Definition~\ref{def:sol_f_g}.
 \end{proof}

 \begin{rem}
  Note that if the denominator on the RHS of \eqref{eq:A_z_formula} vanishes for some $n\in\Z$, then it is not the case for $n+1$. Indeed, if
  \[
   \mathfrak{F\!}\left(\left\{ \frac{\gamma_{k}^{\,2}}{\lambda_{k}-z}\right\} _{k=-\infty}^{n}\right)=\mathfrak{F\!}\left(\left\{ \frac{\gamma_{k}^{\,2}}{\lambda_{k}-z}\right\} _{k=-\infty}^{n+1}\right)=0,
  \]
  for some $n\in\Z$, then 
  \[
   \mathfrak{F\!}\left(\left\{ \frac{\gamma_{k}^{\,2}}{\lambda_{k}-z}\right\} _{k=-\infty}^{n}\right)=0, \quad \forall n\in\Z,
  \]
  which one deduces from the recurrence~\eqref{eq:F_recur_-inf}. However, this would contradict the second limit relation in \eqref{eq:limits_F_1}.
 \end{rem}

 Now we are in position to prove the main result of this section. 
 
 \begin{thm}\label{thm:main}
  Let $\lambda:\Z\to\C$ and $w:\Z\to\C\setminus\{0\}$ be such that \eqref{eq:assum_sum_w} holds for at least one $z_{0}\in\C_{0}^{\lambda}$. Further, assume 
  that $F_{\mathcal{J}}$ does not identically vanish on $\C_{0}^{\lambda}$. Then it holds:
  \begin{enumerate}[{\upshape i)}]
   \item The matrix $\mathcal{J}$ determines the Jacobi operator uniquely, i.e., $J_{\min}=J_{\max}=:J$.
   \item One has equalities
   \[
    \mathfrak{Z}(\mathcal{J})=\spec_{p}(J)\setminus\der(\lambda)=\spec(J)\setminus\der(\lambda).
   \]
   \item The resolvent set $\rho(J)$ of $J$ is nonempty and the Green function 
   \[
   G_{i,j}(z):=\langle e_{i},(J-z)^{-1}e_{j}\rangle, \quad i,j\in\Z,\; z\in\rho(J),
   \]
   is given by the formula
   \begin{align}
    &\hskip10ptG_{i,j}(z)=-\frac{1}{w_{M}}\left(\prod_{k=m}^{M}\frac{w_{k}}{z-\lambda_{k}}\right)\nonumber\\
    &\hskip10pt\times\mathfrak{F\!}\left(\left\{ \frac{\gamma_{k}^{\,2}}{\lambda_{k}-z}\right\} _{k=-\infty}^{m-1}\right)
    \mathfrak{F\!}\left(\left\{ \frac{\gamma_{k}^{\,2}}{\lambda_{k}-z}\right\} _{k=M+1}^{\infty}\right)\Bigg/
    \mathfrak{F\!}\left(\left\{ \frac{\gamma_{k}^{\,2}}{\lambda_{k}-z}\right\} _{k=-\infty}^{\infty}\right)\!,\nonumber\\ 
    \label{eq:Green_func_F}
   \end{align}
   for all $z\in\rho(J)\setminus\der(\lambda)$, where $m:=\min(i,j)$ and $M:=\max(i,j)$.
   (If $z\in\Ran(\lambda)\setminus\der(\lambda)$, the RHS of~\eqref{eq:Green_func_F} is to be understood as the corresponding limit value.)
  \end{enumerate}
 \end{thm}
 
 \begin{proof}
  We divide the proof into 2 parts.
  
  1) Let $z\notin\mathfrak{Z}(\mathcal{J})$. Such $z$ exists due to the assumption that $F_{\mathcal{J}}\neq0$ on $\C_{0}^{\lambda}$. For the sake of simplicity,
  we will further assume that this $z$ belongs to $\C_{0}^{\lambda}$. The purpose of this assumption is to avoid complicated expressions caused by
  the necessary regularization if $z\in\Ran(\lambda)\setminus\der(\lambda)$. However, the idea of the proof remains completely the same.
  
  Let $\mathcal{G}(z)$ be the doubly infinite matrix whose elements are given by the RHS of \eqref{eq:Green_func_F}.
  First, by employing ideas similar to those used in the proof of Lemma \ref{lem:sum_P_at_inf}, one shows that
  \[
   \frac{1}{|w_{M}|}\left|\prod_{k=m}^{M}\frac{w_{k}}{z-\lambda_{k}}\right|\leq C_{1}(z) 2^{-|i-j|},
  \]
  for some constant $C_{1}(z)>0$ and all $i,j\in\Z$, where $m:=\min(i,j)$ and $M:=\max(i,j)$. Next, the assumptions also guarantee that
  the expression
  \[
   \left|\mathfrak{F\!}\left(\left\{ \frac{\gamma_{k}^{\,2}}{\lambda_{k}-z}\right\} _{k=-\infty}^{m-1}\right)
    \mathfrak{F\!}\left(\left\{ \frac{\gamma_{k}^{\,2}}{\lambda_{k}-z}\right\} _{k=M+1}^{\infty}\right)\Bigg/
    \mathfrak{F\!}\left(\left\{ \frac{\gamma_{k}^{\,2}}{\lambda_{k}-z}\right\} _{k=-\infty}^{\infty}\right)\right|
  \]
  is majorized by a constant $C_{2}(z)>0$ for all $i,j\in\Z$. Altogether, one has
  \begin{equation}
   |\mathcal{G}_{i,j}(z)|\leq C(z) 2^{-|i-j|}, \quad \forall i,j\in\Z,
  \label{eq:G_ij_bound}
  \end{equation}
  where the constant $C(z)=C_{1}(z)C_{2}(z)>0$ is independent of the indices $i$ and $j$.
  
  Now, we may introduce the operator $R(z)$ defined as
  \begin{equation}
   R(z):=\sum_{s=-\infty}^{\infty}R(z;s),
  \label{eq:def_op_R}
  \end{equation}
  where $R(z;s)$ is the bounded operator acting on $\ell^{2}(\Z)$ determined by its matrix entries $R_{i,j}(z;s):=\delta_{i,j+s}\mathcal{G}_{i,j}(z)$, for $i,j\in\Z$.
  By \eqref{eq:G_ij_bound}, the norm of $R(z;s)$ satisfies
  \[
   \|R(z;s)\|=\sup_{i-j=s} |\mathcal{G}_{i,j}(z)| \leq C(z) 2^{-|s|}.
  \]
  Consequently, the series \eqref{eq:def_op_R} converges in the operator norm and $R(z)$ is well defined bounded operator on $\ell^{2}(\Z)$ whose
  matrix in the standard basis coincides with~$\mathcal{G}(z)$.
  
  Note that
  \[
   \mathcal{G}_{i,j}(z)=-\frac{1}{F_{\mathcal{J}}(z)}\begin{cases}
                                                      f_{i}(z)g_{j}(z),& \mbox{ for } i\geq j,\\
                                                      f_{j}(z)g_{i}(z),& \mbox{ for } i\leq j,
                                                     \end{cases}
  \]
  where $f(z)$ and $g(z)$ are given in Definition \ref{def:sol_f_g}. Since, according to Proposition~\ref{prop:sol_f_g}, 
  $f(z)$ and $g(z)$ are solutions of the eigenvalue equation $\mathcal{J}u=zu$, one readily verifies that on the 
  level of formal matrix product
  \begin{equation}
   (\mathcal{J}-z)\mathcal{G}(z)=\mathcal{G}(z)(\mathcal{J}-z)=I. 
  \label{eq:G_J_inverse}
  \end{equation}
  By inspection of domains, the above equalities yield the operators $J_{\max}-z$ and $R(z)$ are mutually inverse and hence 
  \begin{equation}
  R(z)=(J_{\max}-z)^{-1}.   
  \label{eq:R_eq_reso_J_max}
  \end{equation}
  Consequently, $z\in\rho(J_{\max})$ and we have shown that $\spec(J_{\max})\setminus\der(\lambda)\subset\mathfrak{Z}(\mathcal{J})$.
  Taking also into account \eqref{eq:Z_in_spec_p_J_max}, we get
  \[
   \mathfrak{Z}(\mathcal{J})\subset\spec_{p}(J_{\max})\setminus\der(\lambda)\subset\spec(J_{\max})\setminus\der(\lambda)
   \subset\mathfrak{Z}(\mathcal{J}).
  \]
  Consequently, if we show that the claim~(i) holds true, i.e., $J_{\min}=J_{\max}$, the theorem is proved.
  
  2) We show that the assumption of the existence of $z\in\C_{0}^{\lambda}$ such that $F_{\mathcal{J}}(z)\neq0$
  implies $J_{\min}=J_{\max}$, indeed. It follows from \eqref{eq:J_max_rel_J_min} that
  \begin{equation}
   \Ker(J_{\max}-z)=\mathcal{C}\Ker((J_{\min}-z)^{*})
  \label{eq:auxiliary_Ker_rel_min_max}   
  \end{equation}
  and 
  \begin{equation}
   \Ran(J_{\max}-z)=\mathcal{C}\Ran((J_{\min}-z)^{*}),
  \label{eq:auxiliary_Ran_rel_min_max}   
  \end{equation}
  for all $z\in\C$.
  
  Let $z\in\C_{0}^{\lambda}$ be such that $F_{\mathcal{J}}(z)\neq0$. Then $z\in\rho(J_{\max})$, as has already been proved 
  in the first part of the proof. Further, by using the second equation in \eqref{eq:G_J_inverse}, one verifies that 
  \[
   R(z)(J_{\min}-z)=I\upharpoonleft\Dom(J_{\min}).
  \]
  Hence $J_{\min}-z$ is injective and its left-inverse is $R(z)$. 
  Further, it follows from \eqref{eq:auxiliary_Ran_rel_min_max} that $\Ran(J_{\min}-z)$ is a closed subspace,
  see \cite[Chp.~IV, Thm.~5.13]{Kato}, and one has
  \[
   \Ran(J_{\min}-z)=\left[\Ker((J_{\min}-z)^{*})\right]^{\perp}=\ell^{2}(\Z).
  \]
  The second equality in the above equation holds since $\Ker((J_{\min}-z)^{*})=\{0\}$, which follows from \eqref{eq:auxiliary_Ker_rel_min_max} and the
  injectivity of $J_{\max}-z$. Thus, $J_{\min}-z$ is an invertible operator with a bounded inverse
  that has to coincide with its left-inverse $R(z)$. Taking into account \eqref{eq:R_eq_reso_J_max}, we obtain
  \[
   (J_{\min}-z)^{-1}=R(z)=(J_{\max}-z)^{-1},
  \]
  which implies, in particular, that $\Dom(J_{\min})=\Dom(J_{\max})$. 
 \end{proof}
 
 \begin{rem}\label{rem:two_assum_comment}
  Theorem~\ref{thm:main} has been derived under two assumptions:
  
  \smallskip\noindent (i)~The convergence condition (\ref{eq:assum_sum_w}) is fulfilled for at least one $z_{0}\in\mathbb{C}_{0}^{\lambda}$.
  
  \smallskip\noindent (ii)~The function $F_{\mathcal{J}}(z)$ does not vanish identically on $\mathbb{C}_{0}^{\lambda}$.
  
  The assumption (i) is necessary for the definition of the characteristic function $F_{\mathcal{J}}$ and is essential. The assumption (ii) guarantees
  that the matrix $\mathcal{J}$ determines the unique Jacobi operator. It might seem hard to decide if the assumption (ii) is fulfilled or not.
  Let us point out that (ii), if assumed jointly with (i), is not very restrictive and, in applications, it is usually satisfied since $\der(\lambda)$ is typically an empty, 1-point, 
  or 2-point set. 
  
  Indeed, (ii) is automatically fulfilled if (i) holds and the $\Ran\lambda$ is contained in a sector $z_{0}+\{z\in\C \mid |\mbox{arg}\, z -\theta_{0}|\geq c>0\}$
  for some $z_{0}\in\C$ and $\theta_{0}\in[0,2\pi)$. Then the half-line $\mathcal{L}=z_{0}+\{re^{\ii\theta_{0}} \mid r>0\}$ is contained in $\C_{0}^{\lambda}$
  and $1/|\lambda_{k}-z|$ tends to $0$ monotonically, for all $k\in\Z$, as $z\to\infty$, $z\in\mathcal{L}$. 
  Similarly as in Remark~\ref{rem:def_F}, one derives the estimate
  \[
  \left|F_{\mathcal{J}}(z)-1\right|\leq\exp\!\left(\sum_{k=-\infty}^{\infty}\left|\frac{w_{k}^{\,2}}{(\lambda_{k}-z)(\lambda_{k+1}-z)}\right|\right)-1, \quad \forall z\in\C_{0}^{\lambda}.
  \]
  It follows that $F_{\mathcal{J}}(z)\to1$ as $z\to\infty$, $z\in\mathcal{L}$, and hence (ii) is satisfied. 
  Note also that $\Ran\lambda$ is contained in a sector, if it is contained in a half-plane or, in particular, 
  if it is real.
  \end{rem}
 
 \begin{cor}\label{cor:compact_resolvent}
  Suppose, in addition to the assumptions of Theorem~\ref{thm:main}, that $|\lambda_{n}|\to\infty$ as $n\to\pm\infty$. Then
  \begin{equation}
   \mathfrak{Z}(\mathcal{J})=\spec(J)=\spec_{p}(J)
   \label{eq:Z_eq_spec_der_empty}
  \end{equation}
  and the resolvent $(J-z)^{-1}$ is compact for all $z\in\rho(J)$. Moreover, if there exists $p\geq1$ such that 
  \begin{equation}
   \sum_{|n|>n_{0}}\frac{1}{|\lambda_{n}|^{p}}<\infty,
   \label{eq:recip_lambda_in_l_p}
  \end{equation} 
  for some $n_{0}\in\N$, then $(J-z)^{-1}$ belongs to the Schatten--von Neumann class~$\mathcal{S}_{p}$.
 \end{cor}
 
 \begin{proof}
  Since $|\lambda_{n}|\to\infty$ for $n\to\pm\infty$, $\der(\lambda)=\emptyset$. Consequently, 
  \eqref{eq:Z_eq_spec_der_empty} is a particular case of the statement~(ii) of Theorem \ref{thm:main}.
  
  Let $z\in\rho(J)=\C\setminus\mathfrak{Z}(\mathcal{J})$ be such that $z\in\C_{0}^{\lambda}$.
  Such $z$ exists since $F_{\mathcal{J}}$ does not vanish identically on $\C_{0}^{\lambda}$ by assumptions.
  By a slight refinement of the estimate \eqref{eq:G_ij_bound}, one derives that
  \begin{equation}
    |\mathcal{G}_{i,j}(z)|\leq \frac{C_{2}(z)}{|z-\lambda_{i}|^{1/2}|z-\lambda_{j}|^{1/2}}2^{-|i-j|}, \quad \forall i,j\in\Z.
  \label{eq:G_ij_bound_refine}
  \end{equation}
  This implies that, for $R(z;s)$ defined in \eqref{eq:def_op_R}, $R(z;s)_{i,j}\to0$ as $i,j\to\pm\infty$ with $i-j=s\in\Z$ being fixed. Consequently, $R(z;s)$ is compact for
  all $s\in\Z$ and, since the series \eqref{eq:def_op_R} converges in the operator norm, $R(z)$ is also compact.
  
  By assuming \eqref{eq:recip_lambda_in_l_p} additionally, we show that $R(z)\in\mathcal{S}_{p}$. The operator $R(z;s)(R(z;s))^{*}$
  is a diagonal operator with diagonal elements $|\mathcal{G}_{i,i+s}(z)|^{2}$, $i\in\Z$. Hence, the numbers $|\mathcal{G}_{i,i+s}(z)|$, 
  $i\in\Z$, are singular values of the compact operator $R(z;s)$. Taking into account \eqref{eq:G_ij_bound_refine} and using the Cauchy--Schwarz inequality, 
  one obtains the estimation for the $p$-th power of the $p$-th Schatten--von Neumann norm of $R(z;s)$ in the form:
  \begin{align*}
   \|R(z;s)\|_{p}^{p}&\leq 2^{-|s|}C_{2}(z)\sum_{i=-\infty}^{\infty}|z-\lambda_{i}|^{-p/2}|z-\lambda_{i+s}|^{-p/2}\\
   &\leq 2^{-|s|}C_{2}(z)\sum_{i=-\infty}^{\infty}|z-\lambda_{i}|^{-p}.
  \end{align*}
  The assumption \eqref{eq:recip_lambda_in_l_p} guarantees the series on the RHS of the above estimate converges. 
  Moreover, it also follows that the series \eqref{eq:def_op_R} converges in $\mathcal{S}_{p}$ and hence 
  $R(z)\in\mathcal{S}_{p}$.
 \end{proof}

 \section{The characteristic function and the algebraic multiplicity}\label{sec:multiplicity}
 
  The main aim of this section is to prove that the order of a zero of the characteristic function $F_{\mathcal{J}}$ 
  coincides with the algebraic multiplicity of the corresponding eigenvalue. It turns out that the geometric multiplicity 
  of an eigenvalue of the Jacobi operator $J$ under investigation is always one. Consequently, knowledge regarding the algebraic
  multiplicities of eigenvalues is straightforwardly connected to a possible similarity of $J$ to a diagonal operator. 
  Namely, simplicity of all zeros of $F_{\mathcal{J}}$ is a necessary condition for the possible diagonalizability of $J$.
 
 \subsection{Preliminaries}
 
  If $A$ is a closed operator acting on a Hilbert space $\mathcal{H}$ and $z\in\spec_{p}(A)$, we denote by $\nu_{g}(z):=\dim\Ker(A-z)$ 
  the geometric multiplicity of $A$. Further, if $z$ is an isolated point of $\spec(A)$, the algebraic multiplicity of $z$ is defined 
  as $\nu_{a}(z):=\dim\Ran P_{z}$, where
  \[
  P_{z}:=-\frac{1}{2\pi\ii}\oint_{\gamma_{z}}(A-\xi)^{-1}\dd\xi
  \]
 is the Riezs spectral projection and $\gamma_{z}$ is a positively oriented Jordan curve located in $\rho(A)$ that has 
 $z$ the only spectral point of $A$ in its interior. Recall that $\nu_{a}(z)$ coincides with the dimension 
 of the space of generalized eigenvectors
 \[
 \mathcal{M}:=\{v\in\mathcal{H} \mid (A-z)^{n}v=0 \; \mbox{ for some } n\in\N\},
 \]
 see, for example, \cite[Sec.~XII.2]{ReedSimon4}.
 
 Let us assume that $J_{\min}=J_{\max}=:J$. Note that $\dim\Ker(J-z)^{n}\leq2n$, for any $z\in\C$ and $n\in\N$, since $\Ker(J-z)^{n}$ is a subspace of the space of solutions of the difference equation $(\mathcal{J}-z)^{n}u=0$
 which is of order $2n$. Clearly, $\Ker(J-z)^{n}$ is a subspace of $\Ker(J-z)^{n+1}$ for all $n\in\N$. For $z\in\C$ fixed, let us denote by $\mathcal{M}_{1}:=\Ker(J-z)$ and by
 $\mathcal{M}_{n}$ the orthogonal complement of $\Ker(J-z)^{n-1}$ into $\Ker(J-z)^{n}$, for $n>1$. Hence we have
 \[
 \Ker(J-z)^{n+1}=\Ker(J-z)^{n} \oplus \mathcal{M}_{n+1}, \; \mbox{ for } n\in\N, 
 \]
 and
 \[
  \Ker(J-z)^{n}=\bigoplus_{k=1}^{n}\mathcal{M}_{k} \quad \mbox{ and } \quad \mathcal{M}=\bigoplus_{n=1}^{\infty}\mathcal{M}_{n}.
 \]
 Note that, if $\dim\mathcal{M}_{n}=0$ for some $n\in\N$, then $\dim\mathcal{M}_{k}=0$ for all $k\geq n$.
 
 \begin{lem}\label{lem:dim_M}
  Let $J_{\min}=J_{\max}=:J$ with $\rho(J)\neq\emptyset$, then $\dim\mathcal{M}_{n}\leq1$ for all $n\in\N$ and $z\in\C$.
 \end{lem}
 
 \begin{proof}
  First, we show that $\dim\mathcal{M}_{1}\leq1$. There is a standard result that goes back to Wall, see \cite[Thm.~22.1]{Wall}, which shows that the following claims are equivalent:
  
  \smallskip\noindent (i)~The second-order difference equation $(\mathcal{J}-z)u=0$ has two linearly independent solutions belonging to $\ell^{2}(\Z)$ for \emph{one} $z\in\C$.
  
  \smallskip\noindent (ii)~The second-order difference equation $(\mathcal{J}-z)u=0$ has two linearly independent solutions belonging to $\ell^{2}(\Z)$ for \emph{all} $z\in\C$.
  
  \smallskip\noindent Note that $\dim\mathcal{M}_{1}\leq2$. If $\dim\mathcal{M}_{1}=2$, then by the above equivalence, $\dim\Ker(J-z)=2$ for all $z\in\C$ which contradicts the assumption $\rho(J)\neq\emptyset$.
  Hence $\dim\mathcal{M}_{1}\leq1$.
  
  Second, we show that $\dim\mathcal{M}_{n}\leq1$ for all $n>1$. For a contradiction, suppose $\dim\mathcal{M}_{n+1}\geq2$ for some $n\in\N$.
  Thus, there are two linearly independent vectors $\{f,g\}\subset\mathcal{M}_{n+1}$.
  One has $(J-z)^{n}f,(J-z)^{n}g\in\mathcal{M}_{1}$ and these vectors are linearly dependent because $\dim\mathcal{M}_{1}\leq1$. Thus, there exist $\alpha,\beta\in\C$ such that
  $\alpha\neq0$ or $\beta\neq0$ and $\alpha (J-z)^{n}f+\beta(J-z)^{n}g=0$. Then $(J-z)^{n}(\alpha f+\beta g)=0$, and so $(\alpha f+\beta g)\in\Ker(J-z)^{n}$. Since $\Ker(J-z)^{n}$ and 
  $\mathcal{M}_{n+1}$ are mutually orthogonal, we conclude that $\alpha f+\beta g=0$, which is a contradiction with the linear independence of $\{f,g\}$.
 \end{proof}
 
 Further, we will need the following purely algebraic statement.
 
 \begin{lem}\label{lem:annihil}
  Let $\{f^{(m)}\}_{m\in\N_{0}}$ be a sequence of elements of the space of complex sequences such that
  \begin{equation}
   \mathcal{J}f^{(0)}=0 \quad\mbox{ and }\quad \mathcal{J}f^{(m)}=mf^{(m-1)}, \; \mbox{ for } m\in\N,
  \label{eq:annihil_eq}
  \end{equation}
  and $f^{(0)}\neq0$. Then $\{f^{(m)}\}_{m\in\N_{0}}$ is linearly independent. In addition, if $\{g^{(m)}\}_{m\in\N_{0}}$
  is another sequence satisfying the equations~\eqref{eq:annihil_eq} (with $f$ being replaced by $g$) and
  \begin{equation}
   C^{(k)}(f,g):=\sum_{j=0}^{k}\binom{k}{j}C\left(f^{(j)},g^{(k-j)}\right)\!,
  \label{eq:def_C^k}
  \end{equation}
  where $C_{n}(f,g):=f_{n}g_{n+1}-f_{n+1}g_{n}$, $\forall n\in\Z$, then the implication
  \begin{equation}
   C^{(\ell)}(f,g)=0, \;\forall \ell\in\{0,\dots,k\} \; \Rightarrow \; g^{(k)}\in\spn\{f^{(0)},\dots,f^{(k)}\}
  \label{eq:vanish_impl_spans}
  \end{equation}
  holds for any $k\in\N_{0}$.
 \end{lem}
 
 \begin{proof}
  First, we verify the linear independence of the set $\{f^{(m)}\}_{m\in\N_{0}}$.
  Suppose, for a contradiction, that $\{f^{(m)}\}_{m\in\N_{0}}$ is linearly dependent. Take $m_{0}\in\N$ such that $\{f^{(j)} \mid 1\leq j\leq m_{0}\}$ 
  is linearly dependent but $\{f^{(j)} \mid 1\leq j< m_{0}\}$   is linearly independent. Such index $m_{0}$ exists since  $f^{(0)}\neq0$.
  Hence there exist numbers $\alpha_{0},\alpha_{1},\dots,\alpha_{m_{0}}\in\C$, $\alpha_{m_{0}}\neq0$, such that
  \[
  \sum_{j=0}^{m_{0}}\alpha_{j}f^{(j)}=0.
  \]
  By applying $\mathcal{J}$ to both sides of the above equation and using \eqref{eq:annihil_eq}, one gets
  \[
  \sum_{j=1}^{m_{0}}j\alpha_{j}f^{(j-1)}=0.
  \]
  The linear independence of $\{f^{(j)} \mid 1\leq j< m_{0}\}$ implies $\alpha_{j}=0$ for all $j\in\{1,\dots,m_{0}\}$, which is a contradiction with $\alpha_{m_{0}}\neq0$.
  
  Further, the proof of the implication~\eqref{eq:vanish_impl_spans} proceeds by the mathematical induction in~$k$. If $k=0$, then $f^{(0)}$ and $g^{(0)}$ are two solutions
  of the equation $\mathcal{J}u=0$ and their Wronskian vanishes. Indeed, for $n\in\Z$ arbitrary, one has $W(f^{(0)},g^{(0)})=w_{n}C_{n}(f^{(0)},g^{(0)})=0$. Hence, $f^{(0)}$ and $g^{(0)}$
  are linearly dependent.
  
  Let the implication~\eqref{eq:vanish_impl_spans} hold up to $k-1$ for some $k\in\N$ and all couples of sequences $\{f^{(m)}\}_{m\in\N_{0}}$ and $\{g^{(m)}\}_{m\in\N_{0}}$ satisfying the
  assumptions of the statement. Suppose that such $\{f^{(m)}\}_{m\in\N_{0}}$ and $\{g^{(m)}\}_{m\in\N_{0}}$ are given and
  \begin{equation}
  C^{(\ell)}(f,g)=0, \quad \forall\ell\in\{0,\dots,k\}.
  \label{eq:LHS_impli_induction}
  \end{equation}
  Consequently, $f^{(0)}$ and $g^{(0)}$ are linearly dependent and hence $g^{(0)}=\alpha f^{(0)}$ for some $\alpha\in\C$. Put 
  \[
   \tilde{g}^{(m)}:=\frac{1}{m+1}\left(g^{(m+1)}-\alpha f^{(m+1)}\right)\!, \quad \forall m\in\N_{0}.
  \]
  
  Since
  \[
   \mathcal{J}\tilde{g}^{(0)}=\mathcal{J}\left(g^{(1)}-\alpha f^{(1)}\right)=g^{(0)}-\alpha f^{(0)}=0
  \]
  and 
  \[
   \mathcal{J}\tilde{g}^{(m)}=\frac{1}{m+1}\mathcal{J}\left(g^{(m+1)}-\alpha f^{(m+1)}\right)=g^{(m)}-\alpha f^{(m)}=m\tilde{g}^{(m-1)},
  \]
  for $m\in\N$, $\{f^{(m)}\}_{m\in\N_{0}}$ and $\{\tilde{g}^{(m)}\}_{m\in\N_{0}}$ are a new couple of sequences satisfying the assumptions of the statement.
  Further, for $\ell\in\N$, one has
  \begin{align*}
   C^{(\ell-1)}(f,\tilde{g})&=\sum_{j=0}^{\ell-1}\binom{\ell-1}{j}C\left(f^{(\ell-1-j)},\tilde{g}^{(j)}\right)\\
   &=\frac{1}{\ell}\sum_{j=0}^{\ell}\binom{\ell}{j}C\left(f^{(\ell-j)},g^{(j)}-\alpha f^{(j)}\right)\!.
  \end{align*}
  Taking into account that $C(\cdot,\cdot)$ is a bilinear and antisymmetric form, one obtains
  \[
    C^{(\ell-1)}(f,\tilde{g})=\frac{1}{\ell} C^{(\ell)}(f,g), \quad \forall \ell\in\N.
  \]
  The above equality together with~\eqref{eq:LHS_impli_induction} imply $C^{(\ell)}(f,\tilde{g})=0$ for all $\ell\in\{0,\dots,k-1\}$. Consequently, it follows from the induction hypothesis that
  \[
   \tilde{g}^{(k-1)}\in\spn\{f^{(0)},\dots,f^{(k-1)}\}.
  \]
  On the other hand, $\tilde{g}^{(k-1)}=g^{(k)}-\alpha f^{(k)}$ and thus $g^{(k)}\in\spn\{f^{(0)},\dots,f^{(k)}\}$, which concludes the proof.
  \end{proof}

 The last preliminary result is a generalization of Lemma \ref{lem:sum_P_at_inf} with the spectral parameter $z$ restricted to $\C_{0}^{\lambda}$. 
 It will be used later to show that, for the sequences $f(z)$ and $g(z)$ given in Definition~\ref{def:sol_f_g}, all the derivatives of $f(z)$ are 
 square summable at $+\infty$ and all the derivatives of $g(z)$ are square summable at $-\infty$.

 \begin{lem}\label{lem:sum_der_P_at_inf}
   Let $\lambda:\Z\to\C$ and $w:\Z\to\C\setminus\{0\}$ be such that \eqref{eq:assum_sum_w} holds for at least one $z_{0}\in\C_{0}^{\lambda}$. Then, for all $z\in\C_{0}^{\lambda}$ and $k\in\N_{0}$, one has
   \[
    \sum_{n=0}^{\infty}\left|\frac{\dd^{k}}{\dd z^{k}}\mathcal{P}_{n}(z)\right|<\infty \; \mbox{ and } \; \sum_{n=-\infty}^{0}\left|\frac{\dd^{k}}{\dd z^{k}}\frac{1}{w_{n-1}\mathcal{P}_{n-1}(z)}\right|<\infty,
   \]
   where $\mathcal{P}_{n}(z)$ is given by \eqref{eq:def_P_z}.
 \end{lem}
 
 \begin{proof}
  We prove the convergence of the first series only. The verification of the convergence of the second one is analogous.
 
  Let $z\in\C_{0}^{\lambda}$. First, the same arguments as in the proof of Lemma \ref{lem:sum_P_at_inf} show that
  \begin{equation}
   |\mathcal{P}_{n}(z)|\leq D(z) 2^{-n},
  \label{eq:P_n_leq_recip_2^n}
  \end{equation}
  for all $n$ sufficiently large, where $D(z)>0$ is a constant independent of $n$. 
  
  Further, note that
  \[
   \mathcal{P}_{n}'(z)=\xi_{n}(z)\mathcal{P}_{n}(z),
   \]
  where
  \[
   \xi_{n}(z)=\sum_{k=1}^{n}\frac{1}{\lambda_{k}-z}.
  \]
  One readily verifies by mathematical induction in $k\in\N$ that the $k$th derivative of $\mathcal{P}_{n}$ can be expressed as
  \begin{equation}
    \mathcal{P}_{n}^{(k)}=p_{k}\left(\xi_{n},\xi_{n}'\dots,\xi_{n}^{(k-1)}\right)\mathcal{P}_{n}, \; \mbox{ for } k\in\N,
  \label{eq:der_P_eq_pol_xi_P}
  \end{equation}
  where $p_{k}$ is a polynomial in $\xi_{n},\xi_{n}'\dots,\xi_{n}^{(k-1)}$ of the form
  \begin{equation}
   p_{k}\left(\xi_{n},\xi_{n}'\dots,\xi_{n}^{(k-1)}\right)=\sum_{\alpha\in\N_{0}^{k},\, |\alpha|\leq k}m_{\alpha}\prod_{j=1}^{k}\left(\xi_{n}^{(j-1)}\right)^{\alpha_{j}},
  \label{eq:p_k_expr}
  \end{equation}
  with coefficients $m_{\alpha}\in\Z$. Since
  \begin{equation}
   |\xi_{n}^{(k)}(z)|\leq \frac{k!}{\delta^{k+1}}n, \; \mbox{ for } k\in\N_{0},
  \label{eq:xi_n^k_estim}
  \end{equation}
  where $\delta:=\dist(z,\Ran(\lambda))>0$, one deduces from \eqref{eq:der_P_eq_pol_xi_P}, \eqref{eq:p_k_expr}, \eqref{eq:xi_n^k_estim}, and 
  \eqref{eq:P_n_leq_recip_2^n} that there exists $C_{k}(z)>0$ such that
  \[
   \left|\mathcal{P}_{n}^{(k)}(z)\right|\leq C_{k}(z)n^{k}2^{-n},
  \]
  for all $n$ sufficiently large. The above estimate gives a summable majorant for the first series from the statement
  for arbitrary $k\in\N_{0}$.
 \end{proof}

\subsection{The multiplicity theorem}
  
 \begin{thm}\label{thm:multiplicity}
 Let the assumptions of Theorem~\ref{thm:main} be fulfilled. Then the following claims hold true.
  \begin{enumerate}[{\upshape i)}]
   \item For all $z\in\spec_{p}(J)$, $\nu_{g}(z)=1$.
   \item Suppose additionally that the set $\C\setminus\der(\lambda)$ is connected. Then the set $\spec(J)\cap\C_{0}^{\lambda}$ consists of isolated eigenvalues and, if $z\in\C_{0}^{\lambda}\cap\spec(J)$, 
   then $\nu_{a}(z)$ coincides with the order of $z$ as a zero of~$F_{\mathcal{J}}$. Moreover, the space of generalized eigenvectors is spanned by vectors $f(z),f'(z),\dots,f^{(\nu_{a}(z)-1)}(z)$.
  \end{enumerate}
 \end{thm}
 
 \begin{proof}
  According to Theorem~\ref{thm:main}, $J_{\max}=J_{\min}=:J$ and $\rho(J)\neq\emptyset$. Hence the claim (i) follows immediately
  from Lemma~\ref{lem:dim_M}. 
  
  Next, the fact that $\spec(J)\cap\C_{0}^{\lambda}$ contains eigenvalues only follows from the statement~(ii) of Theorem \ref{thm:main}.
  If $\C\setminus\der(\lambda)$ is connected, then $\C_{0}^{\lambda}$ is clearly connected as well.  Further, to arrive at a~contradiction, 
  assume that $\spec(J)\cap\C_{0}^{\lambda}$ has an accumulation point. Then by part~(ii) of Theorem \ref{thm:main},
  the set of zeros of the function $F_{\mathcal{J}}$, which is analytic on $\C_{0}^{\lambda}$, has an accumulation point in $\C_{0}^{\lambda}$.
  Thus $F_{\mathcal{J}}$ has to vanish identically on $\C_{0}^{\lambda}$, a contradiction with the assumption.
  
  Further, let $z_{0}\in\C_{0}^{\lambda}\cap\spec(J)$, then, by Theorem~\ref{thm:main} again, $F_{\mathcal{J}}(z_{0})=0$.
  Let us denote by $n_{0}$ the order of $z_{0}$ as the zero of $F_{\mathcal{J}}$. From the formula \eqref{eq:Green_func_F}, 
  one observes that any zero of $F_{\mathcal{J}}$ is a pole (or removable singularity) of the Green function of order less or equal to 
  the order of the zero. Thus, any matrix element of the resolvent operator $(J-z)^{-1}$ has a pole at $z_{0}$ of order at 
  most $n_{0}$. Consequently, for any $\phi,\psi\in\ell^{2}(\Z)$, the function $z\mapsto\langle\phi,(J-z)^{-1}\psi\rangle$ 
  has a pole at $z_{0}$ of order at most $n_{0}$. It follows from the last assertion that 
  \begin{equation}
  \nu_{a}(z_{0})\leq n_0.
  \label{eq:nu_a_leq_n}
  \end{equation}
  
  Indeed, if $\nu_{a}(z_{0})>n_{0}$, then there exists a Jordan chain of $J-z_{0}$ of the length at least $n_{0}+1$, 
  i.e., there are nonzero vectors $\phi_{0},\phi_{1},\dots,\phi_{n_{0}}\in\Dom J$ such that 
  \[
   (J-z_{0})\phi_{0}=0 \; \mbox{ and } \; (J-z_{0})\phi_{k}=\phi_{k-1}, \; \mbox{ for } k=1,\dots,n_{0}.
  \]
  From the above equations, one deduces that
  \begin{equation}
    (J-z)^{-1}\phi_{k}=\sum_{j=0}^{k}\frac{1}{(z_{0}-z)^{k+1-j}}\phi_{j}, \; \mbox{ for } k=0,\dots,n_{0},
  \label{eq:poles_vec_Jord_chain}
  \end{equation}
  where $z$ is supposed to be in a neighborhood of $z_{0}$ belonging to the resolvent set of $J$ that 
  exists since $z_{0}$ is an isolated eigenvalue of $J$. By putting $k=n_{0}$ in \eqref{eq:poles_vec_Jord_chain},
  one observes that $\langle\phi_{0},(J-z)^{-1}\phi_{n_{0}}\rangle$ has a singularity at $z_{0}$ of order $n_{0}+1$, a contradiction.

  In the remaining part of the proof, we prove that $\{f(z_{0}),f'(z_{0}),\dots,$ $f^{(n_{0}-1)}(z_{0})\}$
  is a linearly independent set of generalized eigenvectors of $J$ corresponding to the eigenvalue $z_{0}$ which, together with 
  \eqref{eq:nu_a_leq_n}, will conclude the proof of the assertion~(ii).
  
  First, we show that $f^{(j)}(z)$ is a square summable sequence at $+\infty$ for arbitrary $z\in\C_{0}^{\lambda}$ and
  $j\in\N_{0}$. To this end, it suffices to note
  that
  \begin{equation}
   \left|\frac{\dd^{i}}{\dd z^{i}}\mathfrak{F\!}\left(\left\{ \frac{\gamma_{k}^{\,2}}{\lambda_{k}-z}\right\}_{k=n}^{\infty}\right)\right|\leq K(z),
  \label{eq:F_der_com_bound}
  \end{equation}
  for all $n\in\N$ and $0\leq i\leq j$, with some $K(z)>0$. This holds true since
  \[
   \lim_{n\to\infty}\mathfrak{F\!}\left(\left\{ \frac{\gamma_{k}^{\,2}}{\lambda_{k}-z}\right\}_{k=n}^{\infty}\right)=1
  \]
  and the convergence is local uniform in $z$ on $\C_{0}^{\lambda}$, as one deduces from the inequality
  \[
   \left|\mathfrak{F\!}\left(\left\{ \frac{\gamma_{k}^{\,2}}{\lambda_{k}-z}\right\}_{k=n}^{\infty}\right)-1\right|
   \leq\exp\!\left(\sum_{k=n}^{\infty}\left|\frac{w_{k}^{\,2}}{(\lambda_{k}-z)(\lambda_{k+1}-z)}\right|\right)-1, \quad \forall z\in\C_{0}^{\lambda},
  \]
  which, in its turn, is obtained similarly as the one from Remark~\ref{rem:def_F}.
  Recalling \eqref{eq:def_f_z} and using~\eqref{eq:F_der_com_bound}, one gets
  \[
   |f_{n}^{(j)}(z)|\leq K(z)\sum_{i=0}^{j}\binom{j}{i}\left|\mathcal{P}_{n}^{(i)}(z)\right|\!, \quad \forall n\in\N.
  \]
  Now, by applying Lemma~\ref{lem:sum_der_P_at_inf}, one concludes that $f^{(j)}(z)$ is a summable and hence also square summable sequence at $+\infty$.
  Analogously, with the aid of Lemma~\ref{lem:sum_der_P_at_inf}, one verifies that $g^{(j)}(z)$ is a square summable sequence at $-\infty$ for all $z\in\C_{0}^{\lambda}$ 
  and $j\in\N_{0}$.
  
  Further, recall Proposition~\ref{prop:sol_f_g}. By differentiating the equation $\mathcal{J}f(z)=zf(z)$ with respect to $z$, one obtains equalities
  \begin{equation}
   (\mathcal{J}-z_{0})f(z_{0})=0 \; \mbox{ and } \; (\mathcal{J}-z_{0})f^{(j)}(z_{0})=jf^{(j-1)}(z_{0}), \quad \forall j\in\N.
   \label{eq:gener_eigenvec_eq}
  \end{equation}
  In addition, $f(z_{0})\neq0$ since it is an eigenvector of $J$. Clearly, the same holds true if $f$ is replaced by $g$.
  Thus, the couple of sequences $\{f^{(j)}(z_{0})\}_{j\in\N_{0}}$ and $\{g^{(j)}(z_{0})\}_{j\in\N_{0}}$ fulfills the assumptions of Lemma~\ref{lem:annihil}
  where $\mathcal{J}$ is replaced by $\mathcal{J}-z_{0}$. According to this Lemma, the set $\{f(z_{0}),f'(z_{0}),\dots,f^{(n_{0}-1)}(z_{0})\}$ is linearly independent. Further, since 
  \[
  F_{\mathcal{J}}(z)=W(f(z),g(z))=w_{n}C_{n}(f(z),g(z)),\quad \forall z\in\C_{0}^{\lambda},
  \]
  for $n\in\Z$ arbitrary, one obtains by differentiation that
  \begin{equation}
  F_{\mathcal{J}}^{(j)}(z)=w_{n}C_{n}^{(j)}(f(z),g(z)), \quad \forall j\in\N_{0},\;\forall n\in\Z, \; \mbox { and }\; \forall z\in\C_{0}^{\lambda},
  \label{eq:F_rel_Casoratian}
  \end{equation}
  where $C^{(j)}(f(z),g(z))$ is as in~\eqref{eq:def_C^k}. Since $z_{0}$ is a zero of~$F_{\mathcal{J}}$ of order $n_{0}$, $F_{\mathcal{J}}^{(j)}(z_{0})=0$ for $0\leq j <n_{0}$, 
  and hence, by~\eqref{eq:F_rel_Casoratian}, $C^{(j)}(f(z_{0}),g(z_{0}))=0$ for $0\leq j <n_{0}$. Thus, Lemma~\ref{lem:annihil} yields
  \begin{equation}
  \spn\left\{f(z_{0}),f'(z_{0}),\dots,f^{(n_{0}-1)}(z_{0})\right\}=\spn\left\{g(z_{0}),g'(z_{0}),\dots,g^{(n_{0}-1)}(z_{0})\right\}\!.
  \label{eq:fg_span_coinc}
  \end{equation}

  Since all the sequences from the span on the LHS of \eqref{eq:fg_span_coinc} are square summable at $+\infty$
  and all the sequences from the span on the RHS of \eqref{eq:fg_span_coinc} are square summable at $-\infty$
  one concludes that
  \[
   f^{(j)}(z_{0})\in\ell^{2}(\Z), \quad \forall j\in\{0,1,\dots,n_{0}-1\}.
  \]
  Finally, by using the equalities from~\eqref{eq:gener_eigenvec_eq}, one gets
  \[
   f^{(j)}(z_{0})\in\Dom J \; \mbox{ and } \; (J-z_{0})^{n_{0}}f^{(j)}(z_{0})=0,
  \]
  for all $j\in\{0,1,\dots,n_{0}-1\}$. Hence $\{f(z_{0}),f'(z_{0}),\dots,f^{(n_{0}-1)}(z_{0})\}$ is a linearly 
  independent set of generalized eigenvectors. 
 \end{proof}
 
 \begin{rem}
  Note that, under the assumptions of Theorem~\ref{thm:multiplicity}, any spectral point of $J$ located in $\C_{0}^{\lambda}$ is an isolated eigenvalue
  whose algebraic multiplicity is finite. Moreover, the proof of Theorem~\ref{thm:multiplicity} together with Lemma~\ref{lem:dim_M} actually show that, 
  for $z_{0}\in\C_{0}^{\lambda}$ a zero of $F_{\mathcal{J}}$ of order $n_{0}$, it holds
  \[
   \mathcal{M}_{j}=\spn\{\hat{f}^{(j-1)}(z_{0})\}, \; \mbox{ for } 1\leq j \leq n_{0},
  \]
  and 
  \[
   \mathcal{M}_{j}=\{0\}, \; \mbox{ for } j>n_{0},
  \]
  where $\{\hat{f}(z_{0}),\hat{f}'(z_{0}),\dots,\hat{f}^{(n_{0}-1)}(z_{0})\}$ is the set of vectors obtained by the application of the Gram--Schmidt orthogonalization 
  procedure to the set $\{f(z_{0}),f'(z_{0}),\dots,f^{(n_{0}-1)}(z_{0})\}$.
 \end{rem}

 The following corollary gives a necessary condition for $J$ to be diagonalizable, i.e., similar to a diagonal operator. The notion of diagonalizability of a non-self-adjoint operator 
 deserves a more detailed explanation. Usually, the operator of the similarity transformation is required to be bounded with a bounded inverse. This yields nontrivial questions
 concerning basiness of the set of eigenvectors. However, these questions are out of the scope of the current paper, and we do not address them here. Let us only mention that 
 the coincidence of algebraic and geometric multiplicity of all eigenvalues is a necessary condition for an operator to be diagonalizable in any reasonable sense. 
 
 \begin{cor}
  Let the assumptions of Theorem~\ref{thm:main} be fulfilled and let $\C\setminus\der(\lambda)$ be connected. If there exists an eigenvalue $z\in\C_{0}^{\lambda}$
  of $J$ such that the corresponding eigenvector $v(z)$ satisfies
  \[
   \sum_{n=-\infty}^{\infty}v_{n}^{2}(z)=0,
  \]
 then $\nu_{a}(z)>\nu_{g}(z)$.
 \end{cor}
 
 \begin{proof}
  By Theorem \ref{thm:multiplicity}, $\nu_{g}(z)=1$. Thus, $v(z)=cf(z)$ for some $c\neq0$. By applying 
  Theorem \ref{thm:main} and Proposition \ref{prop:l2_norm}, we obtain $F_{\mathcal{J}}(z)=F_{\mathcal{J}}'(z)=0$. 
  Consequently, Theorem~\ref{thm:multiplicity} implies that $\nu_{a}(z)\geq2$.
 \end{proof}

 \begin{cor}
  Let $\lambda:\Z\to\R$ and $w:\Z\to\R\setminus\{0\}$ be such that \eqref{eq:assum_sum_w} holds for at least one $z_{0}\in\C_{0}^{\lambda}$ and $\der(\lambda)\neq\mathbb{R}$. 
  Then all the zeros of $F_{\mathcal{J}}$ are real and simple.
 \end{cor}
 
 \begin{proof}
  Since $\Ran\lambda\subset\R$, $F_{\mathcal{J}}$ does not vanish identically on $\C_{0}^{\lambda}$, see Remark~\ref{rem:two_assum_comment}. 
  Further, according to Theorem~\ref{thm:main}, $J_{\max}=J_{\min}=:J$. Since $w$ is also assumed to be real, $J$ is self-adjoint.
  
  Next, note that, if $\der(\lambda)\neq\mathbb{R}$, then $\C\setminus\der(\lambda)$ is connected. Let $z_{0}\in\C_{0}^{\lambda}$ be a zero of~$F_{\mathcal{J}}$. 
  By Theorems~\ref{thm:main} and \ref{thm:multiplicity}, $z_{0}$ is an isolated eigenvalue of the self-adjoint operator $J$ and therefore $z_{0}\in\R$. Moreover, by 
  the self-adjointness of $J$ and the claim~(i) of Theorem~\ref{thm:multiplicity}, one has $\nu_{a}(z_{0})=\nu_{g}(z_{0})=1$.  Hence, according to the claim~(ii) of Theorem~\ref{thm:multiplicity}, 
  $z_{0}$ is a simple zero of $F_{\mathcal{J}}$.
 \end{proof}

\section{Diagonals admitting global regularization and connections with regularized determinants}\label{sec:diag_reg}

Most of the results obtained within Sections \ref{sec:char_func} and \ref{sec:multiplicity} have been derived with
the spectral parameter restricted to the set $\C_{0}^{\lambda}$. For example, the zeros of the characteristic
function $F_{\mathcal{J}}$ coincide with $\spec_{p}(J)\cap\C_{0}^{\lambda}$ provided that the assumptions of
Theorem~\ref{thm:main} hold. One can even go further and relate the points from $\Ran(\lambda)\setminus\der(\lambda)$
to $\spec_{p}(J)$, as has been done in the claim~(ii) of Theorem~\ref{thm:main}. Doing so, one is forced to locally regularize 
the characteristic function by the limit formula
\[
 \lim_{u\to z}(u-z)^{r(z)}F_{\mathcal{J}}(u),
\]
which is well defined for all $z\in\C\setminus\der(\lambda)$ since $F_{\mathcal{J}}$ has a pole of finite order at $z\in\Ran(\lambda)\setminus\der(\lambda)$ less than or 
equal to $r(z)<\infty$ or removable singularity. However, the resulting extended function is no more analytic on the open set where it is defined, i.e., $\C\setminus\der(\lambda)$.

On the other hand, under some additional assumptions concerning the diagonal sequence $\lambda$, one can regularize $F_{\mathcal{J}}$
globally by multiplying $F_{\mathcal{J}}$ by a suitable function having zeros at the points from $\Ran(\lambda)\setminus\der(\lambda)$
with respective multiplicities. It turns out that the resulting function is either entire or analytic everywhere except the origin. We distinguish 3
different situations where the diagonal sequence admits the global regularization of $F_{\mathcal{J}}$ and, in each case,
we formulate a proposition that combines particular results of Theorems \ref{thm:main} and \ref{thm:multiplicity} with the set of treated
spectral points extended to either $\C$, or $\C\setminus\{0\}$. Moreover, we provide an illustrative concrete example to each case and, in two cases, 
we indicate how the regularized characteristic function is related with the theory of regularized determinants \cite{Simon}.

\subsection{The compact case}

In addition to the assumption \eqref{eq:assum_sum_w}, let us assume
\begin{equation}
 \sum_{n=-\infty}^{\infty}|\lambda_{n}|^{p}<\infty,
 \label{eq:assum_compact}
\end{equation}
for some $p\in\N$.
The condition \eqref{eq:assum_compact} implies $\lambda_{n}\to0$, as $n\to\pm\infty$, and hence
$\der(\lambda)=\{0\}$. Note also that, assuming \eqref{eq:assum_compact}, the condition \eqref{eq:assum_sum_w} holds true
for any $z_{0}\neq0$ if and only if $w\in\ell^{2}(\Z)$.

Under the condition \eqref{eq:assum_compact}, we can introduce the functions defined by the Hadamard products
\begin{equation}
 \Phi_{p}^{+}(z):=\prod_{n=1}^{\infty}\left(1-\frac{\lambda_{n}}{z}\right)\exp\left(\sum_{j=1}^{p-1}\frac{1}{j}\left(\frac{\lambda_{n}}{z}\right)^{\! j}\right)\!,
 \label{eq:def_Phi_+}
\end{equation}
\begin{equation}
 \Phi_{p}^{-}(z):=\prod_{n=-\infty}^{0}\left(1-\frac{\lambda_{n}}{z}\right)\exp\left(\sum_{j=1}^{p-1}\frac{1}{j}\left(\frac{\lambda_{n}}{z}\right)^{\! j}\right)\!,
 \label{eq:def_Phi_-}
\end{equation}
and
\begin{equation}
 \Phi_{p}(z):=\Phi_{p}^{-}(z)\Phi_{p}^{+}(z),
 \label{eq:def_Phi}
\end{equation}
for all $z\in\C\setminus\{0\}$. Functions \eqref{eq:def_Phi_+}, \eqref{eq:def_Phi_-}, and \eqref{eq:def_Phi} are well defined and analytic on $\C\setminus\{0\}$, see,
for example, \cite[Chp.~11]{Conway}. In addition, their zeros are values $\lambda_{n}$, for $n\geq1$, $n\leq0$, and $n\in\Z$, respectively. Thus, we can regularize
the functions $f$, $g$, and $F_{\mathcal{J}}$, defined in Definitions~\ref{def:char_fction} and~\ref{def:sol_f_g}, by introducing functions
\begin{equation}
 \tilde{f}(z):=\Phi_{p}^{+}(z)f(z), \quad \tilde{g}(z):=\Phi_{p}^{-}(z)g(z), \quad\mbox{ and }\quad \tilde{F}_{\mathcal{J}}(z):=\Phi_{p}(z)F_{\mathcal{J}}(z).
 \label{eq:f_g_char_func_regularized_compact}
\end{equation}
All the functions $\tilde{f}$, $\tilde{g}$, and $\tilde{F}_{\mathcal{J}}$ are analytic on $\C\setminus\{0\}$.

Now, we can formulate a proposition summarizing particular results from Theorem~\ref{thm:main} and~\ref{thm:multiplicity} adjusted to the Jacobi operator whose
diagonal sequence satisfies~\eqref{eq:assum_compact}. While the first part of the statement is an immediate consequence of Theorem~\ref{thm:main}, the second 
part concerning algebraic multiplicities does not follow from Theorem \ref{thm:multiplicity} readily. However, to verify this statement, one follows the same 
steps as in the proof of Theorem~\ref{thm:multiplicity} with the functions $f$ and $g$ being replaced by their regularized extensions $\tilde{f}$ and 
$\tilde{g}$. For this reason and the sake of brevity, the proof is only indicated.

\begin{prop}\label{prop:main_compact}
 Let $\lambda:\Z\to\C$ and $w:\Z\to\C\setminus\{0\}$ be such that 
 $w\in\ell^{2}(\Z)$ and $\lambda\in\ell^{p}(\Z)$, for some $p\in\N$.
 Then
 \begin{equation}
  \spec(J)=\spec_{p}(J)\cup\{0\}=\{z\in\C\setminus\{0\} \mid \tilde{F}_{\mathcal{J}}(z)=0\}\cup\{0\}.
 \label{eq:spec_zero_char_func_compact}
 \end{equation}
 In addition, the algebraic multiplicity $\nu_{a}(z)$ of a nonzero eigenvalue $z$ of $J$ coincides with the order of $z$ as a zero of $\tilde{F}_{\mathcal{J}}$
 and the space of generalized eigenvectors is spanned by vectors 
  $\tilde{f}(z),\tilde{f}'(z),\dots,\tilde{f}^{(\nu_{a}(z)-1)}(z)$.
\end{prop}

\begin{proof}
 The assumptions imply that $\der(\lambda)=\{0\}$ and the condition \eqref{eq:assum_sum_w} is satisfied for any $z_{0}\neq0$.
 Moreover, by Remark \ref{rem:two_assum_comment}, $F_{\mathcal{J}}\neq0$ on $\C_{0}^{\lambda}$ since $\Ran(\lambda)$ is bounded. Then Theorem~\ref{thm:main}
 tells us that $J_{\min}=J_{\max}=:J$ and
 \[
  \spec(J)\setminus\{0\}=\spec_{p}(J)\setminus\{0\}=\{z\in\C\setminus\{0\} \mid \tilde{F}_{\mathcal{J}}(z)=0\}.
 \]
 To verify~\eqref{eq:spec_zero_char_func_compact}, it suffices to realize that $J$ is compact and hence $0\in\spec(J)$. 
 Indeed, one easily shows that the Jacobi operator $J$ is compact if and only if
 \[
  \lim_{n\to\pm\infty}\lambda_{n}=\lim_{n\to\pm\infty}w_{n}=0.
 \]
 The above equalities are guaranteed by the assumptions $w\in\ell^{2}(\Z)$ and $\lambda\in\ell^{p}(\Z)$.
 
 The remaining part of the statement is to be derived by the same way as Theorem~\ref{thm:main} where the
 solutions $f$ and $g$ are replaced by their regularized extensions $\tilde{f}$ and $\tilde{g}$.
\end{proof}

 Let us remark that although $0\in\spec(J)$, $0$ need not be an eigenvalue of $J$. Note also that the matrix elements 
 of the resolvent, see Theorem \ref{thm:main} part~(iii), can be now written as 
  \begin{equation}
  \mathcal{G}_{i,j}(z)=-\frac{1}{\tilde{F}_{\mathcal{J}}(z)}\begin{cases}
                                                      \tilde{f}_{i}(z)\tilde{g}_{j}(z),& \mbox{ for } i\geq j,\\
                                                      \tilde{f}_{j}(z)\tilde{g}_{i}(z),& \mbox{ for } i\leq j,
                                                     \end{cases}
  \label{eq:Green_func_compact}
  \end{equation}
  for all $z\in\rho(J)$. Further, following the same steps as in the proof of Proposition \ref{prop:l2_norm}
  replacing everywhere $f(z)$, $g(z)$, and $F_{\mathcal{J}}(z)$ by $\tilde{f}(z)$, $\tilde{g}(z)$, and 
  $\tilde{F}_{\mathcal{J}}(z)$, respectively, one arrives at the summation formula
 \begin{equation}
  \sum_{n=-\infty}^{\infty}\tilde{f}_{n}^{2}(z)=\tilde{A}(z)\tilde{F}_{\mathcal{J}}'(z),
  \label{eq:sum_f_tilde_squared}
 \end{equation}
 for any $z\neq0$ such that $\tilde{F}_{\mathcal{J}}(z)=0$, where $\tilde{A}(z)=\tilde{f}_{n}(z)/\tilde{g}_{n}(z)$ 
 for any $n\in\Z$ such that $\tilde{g}_{n}(z)\neq0$.
 
 \begin{rem}\label{rem:reg_det_compact}
 There is a close connection between the regularized characteristic function $\tilde{F}_{\mathcal{J}}$ and the theory
 of regularized determinants, see \cite[Chp.~9]{Simon}. First, note that $J$ can be decomposed as $J=\Lambda+UW+WU^{*}$, 
 where $\Lambda e_{n}=\lambda_{n}e_{n}$, $We_{n}=w_{n}e_{n}$, and $Ue_{n}=e_{n+1}$ for all $n\in\Z$. The assumptions 
 $\lambda\in\ell^{p}(\Z)$ and $w\in\ell^{2}(\Z)$ imply that $\Lambda\in\mathcal{S}_{p}$ and $W\in\mathcal{S}_{2}$.
 Consequently, $J\in\mathcal{S}_{p}+\mathcal{S}_{2}\subset\mathcal{S}_{\max(2,p)}$ since $\mathcal{S}_{p}\subset\mathcal{S}_{q}$ 
 for $1\leq p\leq q \leq\infty$. Thus, if $p\geq2$, the regularized determinant $\det_{p}(1-zJ)$ is well defined and is an entire function of $z$. We will show that
 \begin{equation}
  \tilde{F}_{\mathcal{J}}(z)=\det\!_{p}(1-z^{-1}J), \quad \forall z\in\C\setminus\{0\}.
  \label{eq:reg_char_func_determinant_compact}
 \end{equation}

 Let $P_{N}$ stand for the orthogonal projection on the space $\spn\{e_{n} \mid |n|\leq N\}$. Without loss off generality, we may assume that \eqref{eq:assum_compact}
 holds with $p\geq2$. Then, by the above discussion, $J\in\mathcal{S}_{p}$ and one has $P_{N}JP_{N}\to J$ in $\mathcal{S}_{p}$, as $N\to\infty$.
 Further, with the aid of the formula for the determinant of a tridiagonal matrix \cite[Eq.~(13)]{StampachStovicek13} and 
 \cite[Thm.~9.2(d)]{Simon} one obtains
 \begin{align*}
  &\det\!_{p}\left(1-zP_{N}JP_{N}\right)\\
  &\hskip48pt=\left[\prod_{n=-N}^{N}\left(1-z\lambda_{n}\right)\exp\left(\sum_{j=1}^{p-1}\frac{z^{j}\lambda_{n}^{j}}{j}\right)\right]
  \mathfrak{F\!}\left(\left\{ \frac{z\gamma_{k}^{\,2}}{1-z\lambda_{k}}\right\} _{k=-N}^{N}\right)\!.
 \end{align*}
 Now, it suffices to send $N\to\infty$ in the above formula to verify \eqref{eq:reg_char_func_determinant_compact}, where one has to take into account that $\det_{p}(1+\cdot)$ is a continuous
 functional on $\mathcal{S}_{p}$, see \cite[Thm.~9.2(c)]{Simon}, and the formula~\eqref{eq:lim_calF_FJ}. Having the formula \eqref{eq:reg_char_func_determinant_compact}
 at hand, claims of Proposition~\ref{prop:main_compact}, with the exception of the one about vectors spanning the generalized eigenspace, may be deduced from general results of the 
 theory of regularized determinants \cite{Simon}.
 \end{rem}
 
 Finally, we illustrate the results derived within this subsection on a concrete example. We follow the standard notation for hypergeometric series, the Bessel function
 of the first kind, the gamma, and digamma function as it is used, for example, in \cite{AbramowitzStegun}.
 
 \begin{example}
  For $\alpha\in\C\setminus\Z$ and $\beta\in\C\setminus\{0\}$, put 
  \[
   \lambda_{n}:=\frac{1}{n+\alpha} \quad \mbox{ and } \quad w_{n}:=\frac{\beta}{\sqrt{(n+\alpha)(n+1+\alpha)}},
  \]
  where $n\in\Z$. Then the assumptions of the Proposition \ref{prop:main_compact} hold with $p=2$. The sequence 
  $\gamma:\Z\to\C$ satisfying $\gamma_{n}\gamma_{n+1}=w_{n}$, for all $n\in\Z$, can be chosen as
  \[
   \gamma_{n}=\sqrt{\frac{\beta}{n+\alpha}}, \quad \mbox{ for } n\in\Z,
  \]
  and, for $z\notin\Ran(\lambda)\cup\{0\}$, we have
  \begin{align}
   \mathfrak{F\!}\left(\left\{ \frac{\gamma_{k}^{\,2}}{z-\lambda_{k}}\right\} _{k=n+1}^{\infty}\right)&=
   \mathfrak{F\!}\left(\left\{ \frac{\beta}{1-z(k+\alpha)}\right\} _{k=n+1}^{\infty}\right)\nonumber\\
   &={}_{0}F_{1}\left(-;n+1+\alpha-z^{-1},-\beta^{2}z^{-2}\right)\!.
  \label{eq:F_n_example_compact1}
  \end{align}
  The second equality in~\eqref{eq:F_n_example_compact1} follows from \cite[Eq.~9.1.69]{AbramowitzStegun} 
  \begin{equation}
   {}_{0}F_{1}\left(-;\nu+1,-z^{2}\right)=\Gamma(\nu+1)z^{-\nu}J_{\nu}(2z), \quad z\in\C,\ \nu\notin-\N,
   \label{eq:0F1_to_Bessel}
  \end{equation}
  and the identity
  \begin{equation}
   J_{\nu}(2z)=\frac{z^{\nu}}{\Gamma(\nu+1)}\,\mathfrak{F}\!\left(\left\{ \frac{z}{\nu+k}\right\} _{k=1}^{\infty}\right)\!, \quad z\in\C,\ \nu\notin-\N,
   \label{eq:BesselJ_rel_F}
  \end{equation}
  which was proved in \cite[Ex.~11]{StampachStovicek11}. It is easy to see, using the definition of the hypergeometric series ${}_{0}F_{1}$, 
  that the RHS of \eqref{eq:F_n_example_compact1} tends to 1 as $n\to-\infty$. Hence, we have
  \begin{equation}
   F_{\mathcal{J}}(z)=1, \quad \forall z\notin\Ran(\lambda)\cup\{0\}.
  \label{eq:char_func_eq_1_example_compact1}    
  \end{equation}
  
  It follows immediately from \eqref{eq:char_func_eq_1_example_compact1} and Theorem~\ref{thm:main} that 
  $\Ran\lambda\subset\spec_{p}(J)$. Nevertheless, we compute the regularized characteristic function
  and hence evaluate the Hilbert--Schmidt determinant \eqref{eq:reg_char_func_determinant_compact}. By definition \eqref{eq:def_Phi_+}, we have
  \[
    \Phi_{2}^{+}(z)=\prod_{n=1}^{\infty}\left(1-\frac{1}{z(n+\alpha)}\right)\exp\left(\frac{1}{z(n+\alpha)}\right), \quad z\neq0.
  \]
  Note that the zeros of the entire function $\Phi_{2}^{+}(1/z)$ coincide with those of $1/\Gamma(\alpha+1-z)$ which is an entire function of order $1$. 
  Consequently, by applying the Hadamard factorization theorem, we obtain
  \[
   \frac{1}{\Gamma(\alpha+1-z)}=e^{a+bz}\Phi_{2}^{+}\left(\frac{1}{z}\right)=e^{a+bz}\prod_{n=1}^{\infty}\left(1-\frac{z}{n+\alpha}\right)e^{z/(n+\alpha)},
  \]
  for some $a,b\in\C$. By putting $z=0$ in the above formula, one gets $e^{a}=1/\Gamma(\alpha+1)$, while, by equating coefficients at $z$ on both sides, one computes 
  $b=\psi(\alpha+1)=\Gamma'(\alpha+1)/\Gamma(\alpha+1)$. Thus, we have
  \begin{equation}
   \Phi_{2}^{+}(z)=e^{-z^{-1}\psi(\alpha+1)}\frac{\Gamma(\alpha+1)}{\Gamma\left(\alpha+1-z^{-1}\right)}, \quad z\neq0.
  \label{eq:Phi_2_plus_example_compact1}
  \end{equation}
  
  With the aid of \eqref{eq:Phi_2_plus_example_compact1}, one further derives
  \begin{equation}
   \Phi_{2}^{-}(z)=\prod_{n=0}^{\infty}\left(1-\frac{1}{z(\alpha-n)}\right)\exp\left(\frac{1}{z(\alpha-n)}\right)
   =e^{z^{-1}\psi(-\alpha)}\frac{\Gamma(-\alpha)}{\Gamma\left(-\alpha+z^{-1}\right)},
  \label{eq:Phi_2_minus_example_compact1}
  \end{equation}
  for $z\neq0$. At last, by multiplying \eqref{eq:Phi_2_plus_example_compact1} and \eqref{eq:Phi_2_minus_example_compact1} and using 
  the identities \cite[Eqs.~6.1.17, 6.3.7]{AbramowitzStegun}
  \begin{equation}
   \Gamma(z)\Gamma(1-z)=\frac{\pi}{\sin\pi z} \; \mbox{ and } \; \psi(1-z)-\psi(z)=\pi\cot\pi z, \quad z\in\C\setminus\Z,
  \label{eq:Gamma_Psi_ident}
  \end{equation}
  we obtain the regularized characteristic function in the form
  \begin{equation}
   \tilde{F}_{\mathcal{J}}(z)=\Phi_{2}(z)F_{\mathcal{J}}(z)=\frac{\sin\pi(\alpha-z^{-1})}{\sin\pi\alpha}e^{z^{-1}\pi\cot\pi\alpha}, \quad \forall z\neq0,
  \label{eq:char_func_example_compact1}
  \end{equation}
  where \eqref{eq:char_func_eq_1_example_compact1} was used. Equivalently, according to \eqref{eq:reg_char_func_determinant_compact}, we proved that
  \[
   \det\!_{2}(1-zJ)=\frac{\sin\pi(\alpha-z)}{\sin\pi\alpha}e^{z\pi\cot\pi\alpha}, \quad \forall z\in\C.
  \]
  
  Having \eqref{eq:char_func_example_compact1}, Proposition \ref{prop:main_compact} tells us that
  \[
   \spec(J)=\spec_{p}(J)\cup\{0\}=\Ran\lambda\cup\{0\}=\left\{1/(\alpha+n) \mid n\in\Z\right\}\cup\{0\},
  \]
  and all the nonzero eigenvalues have the algebraic multiplicity equal to $1$. Further, by using \eqref{eq:F_n_example_compact1} and \eqref{eq:Phi_2_plus_example_compact1}, we readily
  compute
  \begin{align*}
   \tilde{f}_{n}(z)=\left(\frac{\beta}{z}\right)^{n}\sqrt{\frac{\alpha+n}{\alpha}}e^{-z^{-1}\psi(\alpha+1)}&\frac{\Gamma(\alpha+1)}{\Gamma(\alpha+1+n-z^{-1})}\\
   &\hskip12pt\times{}_{0}F_{1}\left(-;n+1+\alpha-z^{-1},-\beta^{2}z^{-2}\right)\!,
  \end{align*}
  which, with the aid of~\eqref{eq:0F1_to_Bessel}, can be rewritten in terms of Bessel functions as
  \[
   \tilde{f}_{n}(z)=\left(\frac{\beta}{z}\right)^{-\alpha+z^{-1}}\Gamma(\alpha+1)e^{-z^{-1}\psi(\alpha+1)}
   \sqrt{\frac{\alpha+n}{\alpha}}J_{n-\alpha-z^{-1}}\left(\frac{2\beta}{z}\right)\!, \quad \forall z\neq0.
  \]
  Further, if we put $z=z_{N}:=1/(N+\alpha)$, for $N\in\Z$, in the above formula and
  omit the factor not depending on $n$, we obtain the $n$th entry of the eigenvector corresponding to the eigenvalue $z_{N}$ in the form
  \[
   v_{n}(z_{N})=\sqrt{\alpha+n}J_{n-N}\left(2\beta(N+\alpha)\right)\!, \quad n,N\in\Z.
  \]
  
  The method based on the characteristic function does not tell us whether the points from $\der(\lambda)$ are
  spectral points of $J$ or not. In this example, we know that $0\in\spec(J)$ since $J$ is compact.
  On the other hand, $0$ is not an eigenvalue of $J$. Indeed, since $w_{n}=\beta\sqrt{\lambda_{n}\lambda_{n+1}}$ for all $n\in\Z$,
  the solution of the eigenvalue equation $Ju=0$ can be obtained by solving the second-order difference 
  equation with constant coefficients:
  \[
   \beta v_{n-1}+v_{n}+\beta v_{n+1}=0, \quad n\in\Z,
  \]
  where $v_{n}:=\sqrt{\lambda_{n}}u_{n}$. By inspection of the asymptotic behavior of the solution $v_{n}$ for $n\to\pm\infty$, one concludes that there is no nontrivial square summable solution of $Ju=0$
  for all $\alpha\in\C\setminus\Z$ and $\beta\in\C\setminus\{0\}$.

  Without going into details, let us also remark that
  \begin{align*}
   \tilde{g}_{n}(z)=(-1)^{n-1}\beta^{\alpha-z^{-1}}z^{-\alpha-1+z^{-1}}\Gamma(-\alpha)e^{z^{-1}\psi(-\alpha)}&\sqrt{\alpha(\alpha+n)}\\
   &\hskip12pt\times J_{-n-\alpha+z^{-1}}\left(\frac{2\beta}{z}\right)\!,
  \end{align*}
  for $z\neq0$. With the aid of identity \cite[Eq.~9.1.5]{AbramowitzStegun}
  \begin{equation}
   J_{-n}(z)=(-1)^{n}J_{n}(z),\quad \forall n\in\Z,
  \label{eq:Bessel_alt_int}
  \end{equation}
  one further verifies that
  \[
   \tilde{A}(z_{N})=\frac{\tilde{f}_{n}(z_{N})}{\tilde{g}_{n}(z_{N})}=(-1)^{N}\beta^{2N}z_{N}^{-2N+1}\frac{\Gamma(1+\alpha)}{\Gamma(1-\alpha)}
   e^{-z_{N}^{-1}\left(\psi(\alpha+1)+\psi(-\alpha)\right)}.
  \]
  In addition, one has
  \[
   F_{\mathcal{J}}'(z_{N})=(-1)^{N}\frac{\pi}{\sin\pi\alpha}z_{N}^{-2}e^{z_{N}^{-1}\left(\psi(-\alpha)-\psi(\alpha+1)\right)}.
  \]
  So we can substitute for all the functions into the formula \eqref{eq:sum_f_tilde_squared} which, after some simplifications, leads to the identity
  \begin{equation}
   \sum_{n=-\infty}^{\infty}(\alpha+n)J_{n-N}^{2}\left(2\beta(N+\alpha)\right)=N+\alpha,
  \label{eq:towards_sum_Bessel}
  \end{equation}
  for $\alpha\notin\Z$ and $\beta\neq0$. Since  
  \[
    \sum_{n=-\infty}^{\infty}nJ_{n}^{2}(z)=0, \quad\forall z\in\C,
  \]
  as one deduces with the aid of~\eqref{eq:Bessel_alt_int}, the identity~\eqref{eq:towards_sum_Bessel} 
  yields the well-known summation formula for Bessel functions \cite[Eq.~9.1.76]{AbramowitzStegun}
  \[
   \sum_{n=-\infty}^{\infty}J_{n}^{2}(z)=1,
  \]
  which holds true for all $z\in\C$.
  
  At last, one can make use of \eqref{eq:Green_func_compact} to deduce that, for $i\geq j$, the Green function reads
  \[
   \mathcal{G}_{i,j}(z)=\frac{(-1)^{j+1}\pi\sqrt{(i+\alpha)(j+\alpha)}}{z\sin\pi\left(\alpha-z^{-1}\right)}
   J_{i+\alpha-z^{-1}}\left(\frac{2\beta}{z}\right)J_{-j-\alpha+z^{-1}}\left(\frac{2\beta}{z}\right)\!,
  \]
  where $z\notin1/(\Z+\alpha)\cup\{0\}$.
\end{example}

\subsection{The case of compact resolvent}

In this subsection, we impose an additional assumption to the diagonal sequence of $\mathcal{J}$ by requiring
\begin{equation}
 0\notin\Ran(\lambda) \quad \mbox{ and } \quad \sum_{n=-\infty}^{\infty}\frac{1}{|\lambda_{n}|^{p}}<\infty,
 \label{eq:assum_compact_resolvent}
\end{equation}
for some $p\in\N$. Clearly, \eqref{eq:assum_compact_resolvent} implies $|\lambda_{n}|\to\infty$, as $n\to\pm\infty$, and so
$\der(\lambda)=\emptyset$. Since $0\in\C_{0}^{\lambda}$ the condition \eqref{eq:assum_sum_w} can be replaced by
\[
 \sum_{n=-\infty}^{\infty}\left|\frac{w_{n}^{2}}{\lambda_{n}\lambda_{n+1}}\right|<\infty.
\]
Note also that the assumption $0\notin\Ran(\lambda)$ is not very restrictive since one can always shift the diagonal sequence $\lambda$ by a constant that causes only a shift in the spectral 
parameter.

Under the condition \eqref{eq:assum_compact_resolvent}, the functions
\begin{equation}
 \Psi_{p}^{+}(z):=\prod_{n=1}^{\infty}\left(1-\frac{z}{\lambda_{n}}\right)\exp\left(\sum_{j=1}^{p-1}\frac{1}{j}\left(\frac{z}{\lambda_{n}}\right)^{\! j}\right)\!,
 \label{eq:def_Psi_+}
\end{equation}
\begin{equation}
 \Psi_{p}^{-}(z):=\prod_{n=-\infty}^{0}\left(1-\frac{z}{\lambda_{n}}\right)\exp\left(\sum_{j=1}^{p-1}\frac{1}{j}\left(\frac{z}{\lambda_{n}}\right)^{\! j}\right)\!,
 \label{eq:def_Psi_-}
\end{equation}
and
\begin{equation}
 \Psi_{p}(z):=\Psi_{p}^{-}(z)\Psi_{p}^{+}(z),
 \label{eq:def_Psi}
\end{equation}
are well defined and entire, see \cite[Chp.~11]{Conway}. Thus, to regularize the key functions we put
 \[
 \tilde{f}(z):=\Psi_{p}^{+}(z)f(z), \quad \tilde{g}(z):=\Psi_{p}^{-}(z)g(z), \quad\mbox{ and }\quad \tilde{F}_{\mathcal{J}}(z):=\Psi_{p}(z)F_{\mathcal{J}}(z),
 \]
for all $z\in\C$. With some abuse of notation, we use the same notation as in \eqref{eq:f_g_char_func_regularized_compact}, although the functions $f$, $g$ and $F_{\mathcal{J}}$ are regularized by
different functions \eqref{eq:def_Psi_+}, \eqref{eq:def_Psi_-}, and \eqref{eq:def_Psi}. This should not cause any confusion since we always work with these regularized (tilde) versions
of the functions $f$, $g$ and $F_{\mathcal{J}}$ within the subsection where they are defined exclusively.

\begin{prop}\label{prop:main_compact_resolvent}
 Let $\lambda,w:\Z\to\C\setminus\{0\}$ be such that 
 \[
 \sum_{n=-\infty}^{\infty}\left|\frac{w_{n}^{2}}{\lambda_{n}\lambda_{n+1}}\right|<\infty \quad \mbox{ and } \quad \sum_{n=-\infty}^{\infty}\frac{1}{|\lambda_{n}|^{p}}<\infty, \quad \mbox{ for some } p\geq1.
 \]
 Further, let $\tilde{F}_{\mathcal{J}}$ do not vanish identically on $\C$. Then $J_{\min}=J_{\max}=:J$ and
 \[
  \spec(J)=\spec_{p}(J)=\{z\in\C \mid \tilde{F}_{\mathcal{J}}(z)=0\}.
 \]
 In addition, the algebraic multiplicity $\nu_{a}(z)$ of an eigenvalue $z$ of $J$ coincides with the order of $z$ as a zero of $\tilde{F}_{\mathcal{J}}$
 and the space of generalized eigenvectors is spanned by the vectors $\tilde{f}(z),\tilde{f}'(z),\dots,\tilde{f}^{(\nu_{a}(z)-1)}(z)$.
\end{prop}

\begin{proof}
 The first two claims of the statement follow readily from Corollary~\ref{cor:compact_resolvent}. The part on the algebraic multiplicity
 and the vectors spanning the space of generalized eigenvectors is to be deduced in the similar way as in Theorem~\ref{thm:multiplicity}
 with functions $f$ and $g$ being replaced by their regularized extensions $\tilde{f}$ and~$\tilde{g}$.
\end{proof}

Under the assumptions of Proposition \ref{prop:main_compact_resolvent}, one also has the summation formula for squares of the elements of eigenvectors as well 
as the formula for the Green function which are of the same form as in \eqref{eq:sum_f_tilde_squared} and \eqref{eq:Green_func_compact}.

\begin{rem}
Similarly as in Remark~\ref{rem:reg_det_compact}, there is a connection between $\tilde{F}_{\mathcal{J}}$ and regularized determinants. 
However, this connection is not that straightforward as in the case of compact $J$.

Under the assumptions of Proposition \ref{prop:main_compact_resolvent}, the auxiliary Jacobi operator $A(z)$, determined by
\[
 A(z)e_{n}:=\frac{w_{n-1}}{\sqrt{\lambda_{n-1}\lambda_{n}}}e_{n-1}-\frac{z}{\lambda_{n}}e_{n}+\frac{w_{n}}{\sqrt{\lambda_{n}\lambda_{n+1}}}e_{n+1}, \quad n\in\Z,
\]
is bounded for all $z\in\C$. In fact, the assumptions of Proposition~\ref{prop:main_compact_resolvent} imply 
$A(z)\in\mathcal{S}_{p}+\mathcal{S}_{2}\subset\mathcal{S}_{\max(2,p)}$, for all $z\in\C$. Hence, assuming without loss of generality
that $p\geq2$, $A(z)\in\mathcal{S}_{p}$ for all $z\in\C$. Since $J_{\min}=J_{\max}=:J$, the operator $J$ coincides with the closure of the operator sum $\Lambda+UW+WU^{*}$ where
$\Lambda$, $W$ and $U$ are defined in Remark \ref{rem:reg_det_compact}. Thus formally (not taking care about domains), one has
\[
 J-z=\Lambda^{1/2}(1+A(z))\Lambda^{1/2}, \quad \forall z\in\C.
\]
Consequently, one expects that $z\in\rho(J)$ if and only if $-1\in\rho(A(z))$. However, this equivalence is true, 
indeed, since, by using Definition~\ref{def:char_fction}, one verifies that $F_{\mathcal{J}}(z)=F_{A(z)}(-1)$ for all 
$z\notin\Ran(\lambda)$. The rest then follows from the claim~(ii) of Theorem~\ref{thm:main}.

Since the matrix of $A(z)$ is tridiagonal, one can make use of the formula \cite[Eq.~(13)]{StampachStovicek13} to show that
 \[
  \det\left(1+P_{N}A(z)P_{N}\right)=\left[\prod_{n=-N}^{N}\left(1-\frac{z}{\lambda_{n}}\right)\right]
  \mathfrak{F\!}\left(\left\{ \frac{\gamma_{k}^{\,2}}{\lambda_{k}-z}\right\} _{k=-N}^{N}\right)\!,
 \]
where $P_{N}$ is the projection on $\spn\{e_{n}\mid |n|\leq N\}$. Further, using \cite[Thm.~9.2(d)]{Simon} one obtains
 \begin{align*}
  &\det\!_{p}\left(1+P_{N}A(z)P_{N}\right)\\
  &\hskip42pt=\left[\prod_{n=-N}^{N}\left(1-\frac{z}{\lambda_{n}}\right)  \exp\left(\sum_{j=1}^{p-1}\frac{1}{j}\left(\frac{z}{\lambda_{n}}\right)^{\! j}\right)\right]
  \mathfrak{F\!}\left(\left\{ \frac{\gamma_{k}^{\,2}}{\lambda_{k}-z}\right\}_{k=-N}^{N}\right)\!.
 \end{align*}
Finally, by sending $N\to\infty$ and taking into account \cite[Thm.~9.2(c)]{Simon} together with the formula \eqref{eq:lim_calF_FJ}, one arrives at the identity
\[
 \det\!_{p}\left(1+A(z)\right)=\tilde{F}_{\mathcal{J}}(z), \quad \forall z\in\C.
\]
\end{rem}

Probably the most simple nontrivial example illustrating the situation treated by 
Proposition~\ref{prop:main_compact_resolvent} is the following one.

\begin{example}
 Put $\lambda_{n}:=n$ and $w_{n}:=w\in\C\setminus\{0\}$ for all $n\in\Z$. Then one can take $\gamma_{n}=\sqrt{w}$ 
 and, for $n\in\Z$ and $z\in\C\setminus\Z$, one has
   \begin{equation}
   \mathfrak{F\!}\left(\left\{ \frac{\gamma_{k}^{\,2}}{z-\lambda_{k}}\right\} _{k=n}^{\infty}\right)=
   \mathfrak{F\!}\left(\left\{ \frac{w}{k-z}\right\} _{k=n}^{\infty}\right)={}_{0}F_{1}\left(-;n-z,-w^{2}\right)\!,
  \label{eq:F_n_example_compact_resolvent1}
  \end{equation}
  as it follows from \eqref{eq:BesselJ_rel_F} and \eqref{eq:0F1_to_Bessel}. 
  By sending $n\to-\infty$ in \eqref{eq:F_n_example_compact_resolvent1}, one gets
    \begin{equation}
   F_{\mathcal{J}}(z)=1, \quad \forall z\in\C\setminus\Z.
  \label{eq:char_func_eq_1_example_compact_resolvent1}    
  \end{equation}
  
  Note that the assumptions of Proposition \ref{prop:main_compact_resolvent} are satisfied with $p=2$ (see also Remark \ref{rem:two_assum_comment}) with the only exception that $0\in\Ran(\lambda)$.
  However, this is not an essential obstacle and it can be easily overcome by taking as the regularizing functions
  \[
   \Psi_{2}^{+}(z):=\prod_{n=1}^{\infty}\left(1-\frac{z}{n}\right)e^{z/n}=\frac{e^{\gamma z}}{\Gamma(1-z)},
  \]
  \[
   \Psi_{2}^{-}(z):=z\prod_{n=-\infty}^{-1}\left(1-\frac{z}{n}\right)e^{z/n}=\frac{e^{-\gamma z}}{\Gamma(z)},
  \]
  and 
  \[
   \Psi_{2}(z):=\Psi_{2}^{-}(z)\Psi_{2}^{+}(z)=\frac{1}{\Gamma(1-z)\Gamma(z)}=\frac{\sin\pi z}{\pi},
  \]
  defined for all $z\in\C$. In the above equalities, we have used the first identity from~\eqref{eq:Gamma_Psi_ident} and the well-known 
  Hadamard product formula for the entire function $1/\Gamma(z)$ involving Euler's constant $\gamma$, see \cite[Eq.~6.1.3]{AbramowitzStegun}.
  With this choice of regularizing functions, Proposition \ref{prop:main_compact_resolvent} remains valid in the same form.
  Taking into account \eqref{eq:char_func_eq_1_example_compact_resolvent1}, the regularized characteristic function reads
  \[
   \tilde{F}_{\mathcal{J}}(z)=\Psi_{2}(z)=\frac{\sin\pi z}{\pi}, \quad \forall z\in\C.
  \]
  Consequently, $\spec(J)=\spec_{p}(J)=\Z$ and all the eigenvalues have the algebraic multiplicity equal to 1. Further,
  \[
   \tilde{f}_{n}(z)=\Psi_{2}^{+}(z)f_{n}(z)=(-1)^{n}e^{\gamma z}w^{z}J_{n-z}(2w), \quad \forall z\in\C,
  \]
  and hence the $n$th component of the eigenvector corresponding to the eigenvalue $N\in\Z$ of $J$ can be chosen as
  \[
   v_{n}(N)=(-1)^{n}J_{n-N}(2w).
  \]
\end{example}

\subsection{The combined case}

Finally, we discuss a situation which is a combination of the previous two cases. We suppose that 
there exists $p\geq1$ such that
\begin{equation}
 \sum_{n=1}^{\infty}|\lambda_{n}|^{p}<\infty \quad \mbox{ and } \quad \sum_{n=-\infty}^{0}\frac{1}{|\lambda_{n}|^{p}}<\infty,
 \label{eq:assum_combined}
\end{equation}
where it is assumed that $\lambda_{n}\neq0$ for $n\leq0$. Clearly, $\der(\lambda)=\{0\}$. One can show that, under the assumption \eqref{eq:assum_combined}, the condition
\eqref{eq:assum_sum_w} holds if and only if 
\[
 \sum_{n=1}^{\infty}|w_{n}|^{2}<\infty \quad \mbox{ and } \quad \sum_{n=-\infty}^{0}\left|\frac{w_{n}^{2}}{\lambda_{n}\lambda_{n+1}}\right|<\infty.
\]
The way to regularize $f$, $g$, and $F_{\mathcal{J}}$ is now the following:
 \[
 \tilde{f}(z):=\Phi_{p}^{+}(z)f(z), \quad \tilde{g}(z):=\Psi_{p}^{-}(z)g(z),
 \]
 and
 \[
  \tilde{F}_{\mathcal{J}}(z):=\Phi_{p}^{+}(z)\Psi_{p}^{-}(z)F_{\mathcal{J}}(z),
 \]
for $z\in\C\setminus\{0\}$, where $\Phi_{p}^{+}$ and $\Psi_{p}^{-}$ are given by~\eqref{eq:def_Phi_+} and \eqref{eq:def_Psi_-}, respectively.

Since there is no significant difference in the derivation of the following statement in comparison with Propositions \ref{prop:main_compact} and \ref{prop:main_compact_resolvent},
we omit the proof completely.

\begin{prop}\label{prop:main_combined}
 Let $\lambda,w:\Z\to\C\setminus\{0\}$ be such that
 \[
  \sum_{n=1}^{\infty}|\lambda_{n}|^{p}<\infty \quad \mbox{ and } \quad \sum_{n=-\infty}^{0}\frac{1}{|\lambda_{n}|^{p}}<\infty, \quad \mbox{ for some } p\geq1,
 \]
 and
 \[
 \sum_{n=1}^{\infty}|w_{n}|^{2}<\infty \quad \mbox{ and } \quad \sum_{n=-\infty}^{0}\left|\frac{w_{n}^{2}}{\lambda_{n}\lambda_{n+1}}\right|<\infty.
 \]
  Further, let $\tilde{F}_{\mathcal{J}}$ does not vanish identically on $\C$. Then $J_{\min}=J_{\max}=:J$ and
  \[
  \spec(J)\setminus\{0\}=\spec_{p}(J)\setminus\{0\}=\{z\in\C\setminus\{0\} \mid \tilde{F}_{\mathcal{J}}(z)=0\}.
  \]
  In addition, the algebraic multiplicity $\nu_{a}(z)$ of a nonzero eigenvalue $z$ of $J$ coincides with the order of $z$ as a zero of $\tilde{F}_{\mathcal{J}}$
  and the space of generalized eigenvectors is spanned by the vectors $\tilde{f}(z),\tilde{f}'(z),\dots,\tilde{f}^{(\nu_{a}(z)-1)}(z)$.
\end{prop}

Also formulas \eqref{eq:Green_func_compact} and \eqref{eq:sum_f_tilde_squared} remain valid in the same form under the assumptions of Proposition \ref{prop:main_combined}.

The following example has been treated in \cite[Sec.~5]{StampachStovicek15} in the real case. Let us revisit the example 
allowing the parameters to be complex and using the characteristic function approach presented here. We shall follow the 
standard notation for basic hypergeometric series as in~\cite{GasperRahman}.

\begin{example}
 Let $\lambda_{n}:=q^{n}$ and $w_{n}:=\beta q^{n/2}$, for $n\in\Z$, where $q,\beta\in\C$, $0<|q|<1$, and $\beta\neq0$. Then we can put $\gamma_{n}=\beta^{1/2}q^{(2n-1)/8}$ for all $n\in\Z$.
 
 It has been proved in \cite[Prop.~15]{StampachStovicek15} that 
 \begin{equation}
 \mathfrak{F}\!\left(\left\{ \frac{w}{q^{-(\nu+k)/2}-q^{(\nu+k)/2}}\right\} _{\! k=0}^{\!\infty}\right)=\,_{0}\phi_{1}(;q^{\nu};q,-q^{\nu+1/2}w^{2}),
 \label{eq:calF_eq_q-confluent}
 \end{equation}
 providing $0<q<1$, $w,\nu\in\C$, and $q^{\nu}\notin q^{-\N_{0}}$. However, the corresponding derivation works with no 
 change even if we assume that $q\in\C$, $0<|q|<1$. From \eqref{eq:calF_eq_q-confluent}, one deduces
 \begin{equation}
  \mathfrak{F\!}\left(\left\{ \frac{\gamma_{k}^{\,2}}{z-\lambda_{k}}\right\}_{k=n}^{\infty}\right)\!=\mathfrak{F\!}\left(\left\{ \frac{\beta q^{(2k-1)/4}}{z-q^{k}}\right\} _{k=n}^{\infty}\right)
  \!={}_{0}\phi_{1}\left(-;z^{-1}q^{n};q,-q^{n}z^{-2}\beta^{2}\right)\!.
  \label{eq:to_f_n_ex_comb}
 \end{equation}
 By sending $n\to\infty$ in~\eqref{eq:to_f_n_ex_comb}, one proves in the same way as found in the proof of 
 \cite[Lem.~18]{StampachStovicek15} that
 \[
  F_{\mathcal{J}}(z)=\left(-\beta^{2}z^{-1};q\right)_{\infty}, \quad \forall z\notin q^{\Z}\cup\{0\}.
 \]
 The assumptions of Proposition \ref{prop:main_combined} are fulfilled with $p=1$ and the regularizing functions are
 \[
  \Phi_{1}^{+}(z)=\left(qz^{-1};q\right)_{\infty} \quad \mbox{ and } \quad \Psi_{1}^{-}(z)=\left(z;q\right)_{\infty}.
 \]
 Altogether, one gets 
 \[
  \tilde{F}_{\mathcal{J}}(z)=\left(z,qz^{-1},-\beta^{2}z^{-1};q\right)_{\infty}, \quad \forall z\in\C\setminus\{0\},
 \]
 and, in virtue of~\eqref{eq:to_f_n_ex_comb}, one obtains
 \begin{equation}
  \tilde{f}_{n}(z)=z^{-n}\beta^{n}q^{n(n-1)/4}\left(z^{-1}q^{n+1};q\right)_{\infty}\,_{0}\phi_{1}\left(-;z^{-1}q^{n+1};q,-q^{n+1}z^{-2}\beta^{2}\right)\!,
  \label{eq:f_tilde_ex_comb}
 \end{equation}
 for all $z\in\C\setminus\{0\}$.
 
 According to Proposition~\ref{prop:main_combined}, one has 
 \[
  \spec(J)=\{0\}\cup q^{\Z} \cup (-\beta^{2})q^{\N_{0}},
 \]
 where all the nonzero spectral points are eigenvalues with the corresponding eigenvector being determined by \eqref{eq:f_tilde_ex_comb} where $z$ is replaced by the respective eigenvalue.
 It can be even shown that $0$ is not an eigenvalue of $J$, see the proof of \cite[Prop.~24]{StampachStovicek15}. If $-\beta^{2}\notin q^{\Z}$, then all the eigenvalues of $J$ have
 the algebraic multiplicity equal to $1$, while if $-\beta^{2}\in q^{\Z}$, then there are infinitely many eigenvalues of $J$ with the algebraic multiplicity equal to $2$.
\end{example}

\end{document}